\documentclass[12pt]{amsart}

\usepackage{pinlabel}
\usepackage{graphicx, overpic}
\usepackage{xcolor}
\usepackage[below]{placeins}
\usepackage[colorlinks=true, linkcolor=blue, citecolor=blue]{hyperref}
\usepackage[]{algorithm2e}
\usepackage{comment}

\usepackage[T1]{fontenc}

\usepackage{amsmath,amsthm,amscd,amssymb,eucal}%, epsfig,psfrag,subfigure,amssymb,}

\usepackage{enumerate, amsfonts, latexsym, color, url}
\usepackage{epstopdf}

\usepackage{pinlabel}

\makeatletter

\begin{document}

\newtheorem{theorem}{Theorem}[section]
\newtheorem{lemma}[theorem]{Lemma}
\newtheorem{proposition}[theorem]{Proposition}
\newtheorem{corollary}[theorem]{Corollary}
\newtheorem{conjecture}[theorem]{Conjecture}
\newtheorem{question}[theorem]{Question}
\newtheorem{problem}[theorem]{Problem}
\newtheorem*{claim}{Claim}
\newtheorem*{criterion}{Criterion}
\newtheorem*{equivalence_theorem}{Equivalence Theorem}
\newtheorem*{conditional_equivalence_theorem}{Conditional Equivalence Theorem}
\newtheorem*{structure_theorem}{Structure Theorem}

\theoremstyle{definition}
\newtheorem{definition}[theorem]{Definition}
\newtheorem{construction}[theorem]{Construction}
\newtheorem{notation}[theorem]{Notation}
\newtheorem{object}[theorem]{Object}
\newtheorem{operation}[theorem]{Operation}

\theoremstyle{remark}
\newtheorem{remark}[theorem]{Remark}
\newtheorem{example}[theorem]{Example}

\numberwithin{equation}{subsection}

\newcommand\id{\textnormal{id}}
\newcommand\CW{CaTherine wheel}
\newcommand\HH{\mathbb H}
\newcommand\N{\mathbb N}
\newcommand\Z{\mathbb Z}
\newcommand\Q{\mathbb Q}
\newcommand\R{\mathbb R}
\newcommand\C{\mathbb C}
\newcommand\DD{\mathbb D}
\newcommand\EE{\mathcal E}
\newcommand\SLE{\textnormal{SLE}}
\newcommand\area{\textnormal{area}}
\newcommand\length{\textnormal{length}}
\newcommand\inte{\textnormal{int}}
\newcommand\un{\textnormal{univ}}

\newcommand{\danny}[1]{{\color{olive}{#1}}}

%Ewain's macros

\newcommand{\ewain}[1]{{\color{blue}{#1}}} 
\newcommand{\td}[1]{\textbf{\color{red}{[#1]}}}

\newcommand{\arXiv}[1]{\href{http://arxiv.org/abs/#1}{arXiv:#1}}
\newcommand{\arxiv}[1]{\href{http://arxiv.org/abs/#1}{#1}}

\def\alb#1\ale{\begin{align*}#1\end{align*}}
\def\allb#1\alle{\begin{align}#1\end{align}}
\newcommand{\eqb}{\begin{equation}}
\newcommand{\eqe}{\end{equation}}
\newcommand{\eqbn}{\begin{equation*}}
\newcommand{\eqen}{\end{equation*}}

\newcommand{\BB}{\mathbb}
\newcommand{\ol}{\overline}
\newcommand{\ul}{\underline}
\newcommand{\op}{\operatorname}
\newcommand{\la}{\langle}
\newcommand{\ra}{\rangle}
\newcommand{\bd}{\mathbf}
\newcommand{\im}{\operatorname{Im}}
\newcommand{\re}{\operatorname{Re}}
\newcommand{\frk}{\mathfrak}
\newcommand{\eqD}{\overset{d}{=}}
\newcommand{\ep}{\varepsilon}
\newcommand{\rta}{\rightarrow}
\newcommand{\xrta}{\xrightarrow}
\newcommand{\Rta}{\Rightarrow}
\newcommand{\hookrta}{\hookrightarrow}
\newcommand{\wt}{\widetilde}
\newcommand{\wh}{\widehat}
\newcommand{\mcl}{\mathcal}
\newcommand{\pre}{{\operatorname{pre}}}
\newcommand{\lrta}{\leftrightarrow}
\newcommand{\bdy}{\partial}
\newcommand{\rng}{\mathring}
\newcommand{\srta}{\shortrightarrow}
\newcommand{\el}{l} 
\newcommand{\eps}{\varepsilon}
\newcommand{\settminus}{-}

%Remove spacing before \left and \right
\let\originalleft\left
\let\originalright\right
\renewcommand{\left}{\mathopen{}\mathclose\bgroup\originalleft}
\renewcommand{\right}{\aftergroup\egroup\originalright}

%Remove spacing before \sim
\let\originalsim\sim
\renewcommand{\sim}{{\originalsim}}

%Redefine \em 
\DeclareRobustCommand{\em}{\bfseries} 
\renewcommand{\emph}[1]{\textbf{#1}}

\title{CaTherine wheels from trees and Liouville quantum gravity}
\author{Danny Calegari}
\address{University of Chicago \\ Chicago, Ill 60637 USA}
\email{dannyc@uchicago.edu}
\author{Ewain Gwynne}
\address{University of Chicago \\ Chicago, Ill 60637 USA}
\email{ewain@uchicago.edu}
\date{\today}

\begin{abstract}
A CaTherine wheel is a space-filling curve $f : S^1\to S^2$ such that for every closed
interval $J\subset S^1$, $f(J)$ is homeomorphic to a closed disk and $f(\partial J)$ is contained in $\partial f(J)$.
A CaTherine wheel gives rise to a pair of disjoint, dense topological trees in $S^2$ which roughly speaking lie to the left and right of $f$.
We give necessary and sufficient conditions for a topological tree in $S^2$ to arise as one of these trees for some CaTherine wheel $f$.
We apply this result to show that there is a unique CaTherine wheel corresponding to the geodesic tree rooted at $\infty$ for the $\gamma$-Liouville quantum gravity (LQG) metric, for $\gamma \in (0,2)$. 
In other words, we construct the space-filling curve which is the contour exploration of the LQG geodesic tree.  
\end{abstract}

\maketitle
\setcounter{tocdepth}{2}
\tableofcontents

\bigskip
\noindent\textbf{Acknowledgments.} We thank Manan Bhatia and Ino Loukidou for helpful discussions, and Manan Bhatia for helpful comments on an earlier version of this paper. 
D.C.\ was partially supported by NSF grant DMS-1928930 while in residence at SLMath
during winter 2026. E.G.\ was partially supported by NSF grant DMS-2245832.

\section{Introduction}

\subsection{CaTherine wheels}
\label{sec-wheel}

A {\em CaTherine wheel} is a surjective map $f:S^1 \to S^2$ such that for every closed
interval $J\subset S^1$ the image $f(J)$ is homeomorphic to a closed (2-dimensional)
disk, and $f(\partial J)$ is contained in $\partial f(J)$.
In other words, a CaTherine wheel is a space-filling loop in $S^2$, 
viewed modulo changes of time parametrization, which maps each interval to a 
closed topological disk and never enters the interior of its past. 
The name derives from Cannon--Thurston, who famously constructed examples
in the theory of hyperbolic 3-manifolds. In fact CaTherine wheels arise in many
contexts throughout low-dimensional geometry and dynamics:
in hyperbolic 3-manifold theory
as Cannon--Thurston maps for 3-manifolds fibering over the circle
\cite{Cannon_Thurston} or admitting pseudo-Anosov flows without perfect fits
\cite{Fenley}; in holomorphic dynamics as Peano curves associated 
to certain expanding Thurston maps of the 2-sphere \cite{Meyer}
and more generally to certain so-called expanding origamis \cite{Calegari_Herr};
and so on. Random CaTherine wheels arise frequently in probability theory, 
as we discuss in Section~\ref{sec-wheel-prob}. A unified theory
of CaTherine wheels is developed in \cite{Catherine_wheels}.

A CaTherine wheel gives rise to a canonical structure called a {\em zipper}, 
which consists of a pair $Z^\pm \subset S^2$
of disjoint, dense, topological $\R$-trees. Every CaTherine wheel arises 
in an essentially unique way as a limit of a family $f_t:S^1 \to S^2$ 
of embedded Jordan curves in $S^2$ (one calls such a family a
{\em pseudo-isotopy}), and one can think of $Z^\pm$ as the
residue of the family of pairs of Jordan domains separated by $f_t(S^1)$ as $t \to 1$.

It turns out (\cite{Catherine_wheels}, Theorem~3.14) that a 
CaTherine wheel can be recovered from its zipper $Z^\pm$. But what if one forgets
$Z^-$ (say) and retains only $Z^+$? The first main theorem of this paper
(Theorem~\ref{theorem:half_zipper}) gives necessary and sufficient conditions for a
subset $Z \subset S^2$ to arise as $Z^+$ for some CaTherine wheel $f:S^1 \to S^2$,
and shows that such an $f$, if it exists, is unique. Once we have proven this theorem, 
we will apply it to establish the existence of a CaTherine wheel whose half-zipper 
$Z^+$ is the geodesic tree for the $\gamma$-Liouville quantum gravity metric for 
$\gamma \in (0,2)$ (see Theorem~\ref{thm-lqg-wheel}).

\subsection{CaTherine wheels from half-zippers}
\label{sec-tree-wheel}
 
A half-zipper $Z$ is a path-connected subset of $S^2$ satisfying
certain properties. It is the path-connectivity that is important, and therefore
throughout this paper we make the following convention: unless we explicitly say to
the contrary, when we refer to ``components'' we mean {\em path components}, and
when we refer to a {\em cut point} in a path-connected set $X$ we mean a point $p$ in $X$ 
for which $X-p$ has at least two path components. Furthermore, when we discuss
paths we usually mean embedded and unparameterized (though sometimes oriented) paths,
again unless we explicitly say to the contrary.

\begin{definition}[half-zipper]\label{definition:half_zipper}
A subset $Z\subset S^2$ is a {\em half-zipper} if it satisfies the following properties:
\begin{enumerate}
\item{$Z$ is dense in $S^2$;}
\item{there is a unique embedded simple path in $Z$ (modulo time parametrization)
connecting any two distinct points of $Z$, and every point of $Z$ is a cut point of $Z$.}
\end{enumerate}
The cut point condition means that for every $p\in Z$ 
there are $q,r\in Z$ so that the unique path in $Z$ from $q$ to $r$ contains $p$ in its interior.

A half-zipper has {\em short hair} if, for any (equivalently, every) metric $d_0$ on $S^2$ which induces the standard topology, it satisfies:
\vskip 2pt
{\hskip 5pt (3) for every $\epsilon>0$ there is a compact subset
$T$ of $Z$ homeomorphic to a finite tree so that every (path) component of 
$Z-T$ has $d_0$-diameter $<\epsilon$.} 
\end{definition}

\begin{remark}[Topological $\R$-tree]\label{remark:topological_R_tree}
A remark about terminology. A half-zipper $Z$ with short hair may 
be expressed as an increasing union of a
countable sequence of finite trees $T_n$ (Lemma~\ref{lemma:increasing_union}). Here
by `finite tree' we mean the underlying topological space of a finite tree in the
sense of graph theory, i.e.\/ a finite connected graph without cycles; trees in this sense
are sometimes called simplicial trees or polygonal trees when there is ambiguity about the term.
In the path topology (i.e.\/ the topology on $Z$ in which a set is open if its
intersection with every $T_n$ is open in $T_n$) the space $Z$ is what is often
called a {\em topological $\R$-tree}, or just {\em topological tree} to distinguish it from
the more familiar simplicial trees like $T_n$. With its path topology, the inclusion
map $Z \to S^2$ is continuous but not a homeomorphism onto its image. This fact is
not used in the paper. 
\end{remark}

We are now ready to state our first main theorem. 

\begin{theorem}[Wheels from short hair]\label{theorem:half_zipper}
Let $Z$ be a half-zipper with short hair (Definition~\ref{definition:half_zipper}). 
Then there is a unique CaTherine wheel $f:S^1 \to S^2$ with zippers $Z^\pm$ for which $Z=Z^+$.
\end{theorem}
 
The CaTherine wheel $f$ of Theorem~\ref{theorem:half_zipper} can be thought of as the clockwise contour 
(depth-first) exploration of the tree $Z$. We also prove a converse to 
Theorem~\ref{theorem:half_zipper}. 

\begin{proposition}[CaTherine wheels short hair]\label{proposition:wheels_to_short_hair}
Let $f:S^1 \to S^2$ be a CaTherine wheel with associated zipper $Z^\pm \subset S^2$. Then
each of $Z^\pm$ has short hair.
\end{proposition}

Theorem~\ref{theorem:half_zipper} and Proposition~\ref{proposition:wheels_to_short_hair} 
are proven in Section~\ref{section:half_zippers} via a short topological argument based 
on results from \cite{Catherine_wheels}.

\begin{remark}[Weaker notions of half-zipper] \label{remark-weaker}
The reader should take care not to confuse our notion of a half-zipper with 
the following superficially similar notion. Let $\wt Z$ be a topological space with 
the property that there is a unique simple path between any two points of $\wt Z$ and 
every point of $\wt Z$ is a cut point; and let $\phi : \wt Z \to S^2$ is a continuous 
injection for which the image of $\wt Z$ is dense. Then the set $\phi(\wt Z)$ 
is \textbf{not} in general a half-zipper. The reason is that $\phi(\wt Z)$ could 
contain a non-trivial loop: indeed, there could be a path $\sigma : [0,1) \to \wt Z$ 
such that $\phi\circ\sigma$ extends continuously 
to $[0,1]$ with $\phi(\sigma(1)) = \phi(\sigma(0))$. That is, $\phi\circ\sigma$ 
``lands'' at a point of $\phi(\wt Z)$ in the sense of Definition~\ref{definition:landing}. 
The fact that sets $\phi(\wt Z)$ of the above type are not half-zippers is important
to ensure that the CaTherine wheel $f$ arising from a half-zipper as in 
Theorem~\ref{theorem:half_zipper} maps closed intervals to sets which are homeomorphic 
to closed disks (instead of to sets which have local cut points). See in particular 
Lemma~\ref{lemma:no_local_cut_points}.
Some discussion about possible generalizations of the definitions of CaTherine wheels and zippers, where the wheel need not map closed intervals to closed disks, can be found in~\cite[Section 5]{Catherine_wheels}. But, the theory of such generalizations is much less developed than the theory of CaTherine wheels. 
\end{remark}

\subsection{A CaTherine wheel from the Liouville quantum gravity geodesic tree}
\label{sec-lqg-wheel}

In this section, we explain an application of Theorem~\ref{theorem:half_zipper} where the half-zipper $Z$ is taken to be the tree of Liouville quantum gravity geodesics rooted at $\infty$. We first need to explain what this means. We give the minimal amount of background necessary to understand the paper here, and provide references for more details.

The whole-plane {\em Gaussian free field (GFF)} $\Phi$ is in some sense the most canonical notion of a random generalized function (distribution) on $\BB C$. We will not need the precise definition of the Gaussian free field in this paper; we refer to~\cite{shef-gff,bp-lqg-notes,pw-gff-notes} for background on the GFF. 

For a parameter $\gamma \in (0,2)$, {\em Liouville quantum gravity (LQG)} is, heuristically speaking, the random geometry on $\BB C$ described by the Riemannian metric tensor
\eqb \label{eqn-lqg-tensor} 
e^{\gamma \Phi} (dx^2 + dy^2) 
\eqe
where $dx^2 + dy^2$ is the Euclidean metric tensor on $\BB C$. The definition~\ref{eqn-lqg-tensor} does not make literal sense since $\Phi$ does not have well-defined pointwise values. But, one can make sense of various objects associated with~\eqref{eqn-lqg-tensor} via regularization procedures. To do this, one replaces $\Phi$ in~\eqref{eqn-lqg-tensor} by a family of continuous functions $\{\Phi_\ep\}_{\ep > 0}$ which converge to $\Phi$ in the distributional sense as $\ep\to 0$, then takes an appropriate limit of objects associated with $\Phi_\ep$. We emphasize that in general, the arguments to show that the limit exists are non-trivial and specific to the Gaussian free field. 

Particularly relevant for us is the ``Riemannian distance function'' $D  = D_\Phi^{(\gamma)}$ associated with~\eqref{eqn-lqg-tensor}, called the {\em $\gamma$-LQG metric}, which is a random metric on $\BB C$ shown to exist in~\cite{dddf-lfpp,gm-uniqueness}. The metric $D$ induces the Euclidean topology on $\BB C$, but it has very different geometric properties. For example, the Hausdorff dimension of $\BB C$ with respect to $D$ is strictly larger than two~\cite{gp-kpz,dg-lqg-dim}. The LQG metric is uniquely characterized (as a function of $\Phi$) by a list of natural axioms in~\cite{gm-uniqueness}, but we will not need these axioms here. See~\cite{ddg-metric-survey} for a survey of results about the LQG metric. 

One can similarly define the ``volume form'' $\mu$, called the {\em $\gamma$-LQG area measure}, which is a random measure on $\BB C$ obtained as a limit of regularized versions of $e^{\gamma \Phi} \,dx\,dy$, where $dx\,dy$ denotes Lebesgue measure~\cite{kahane,rhodes-vargas-log-kpz,shef-kpz}. See~\cite{bp-lqg-notes} for an exposition of this measure. 

We will be interested in geodesics with respect to $D$. To talk about this, we need to review some concepts from metric geometry. For a path $\sigma : [a,b] \to\BB C$, its {\em $D$-length} is the number
\eqb \label{eqn-length}
  \sup_{T} \sum_{i=1}^{\# T} D(\sigma(t_i) , \sigma(t_{i-1})) 
\eqe 
where the supremum is over all partitions $T : a= t_0 < \dots < t_{\# T} = b$ of $[a,b]$. We say that $\sigma$ is a {\em $D$-geodesic} if for all $a \leq t < s \leq b$, the $D$-length of $\sigma|_{[t,s]}$ is equal to $D(\sigma(t), \sigma(s))$. Almost surely,  for each $z,w\in \BB C$, there exists at least one $D$-geodesic from $z$ to $w$. For a fixed $z,w\in\BB C$, this $D$-geodesic is a.s.\ unique (see, e.g.,~\cite[Lemma 4.2]{ddg-metric-survey}). Note that the geodesics considered in this paper are minimizing geodesics, not just local geodesics. In particular, geodesics are necessarily simple paths.

One of the most important features of the LQG metric is {\em confluence of geodesics}, which roughly speaking says the following (see Propositions~\ref{prop-geo-infty} and~\ref{prop-finite-cross} for some precise statements). For a typical point $z\in\BB C$, any two $D$-geodesics started from $z$ stay together for a non-trivial initial time interval before ``branching off'' toward their respective target points. Consequently, the $D$-geodesics started from $z$ form a tree-like structure. Confluence of geodesics for the LQG metric was first proven in~\cite{gm-confluence}. See Figure~\ref{fig-ball-sim}, left and middle, for simulations of LQG metric balls and geodesics.

\begin{figure}[t]
\begin{center}
\includegraphics[width=0.6\textwidth]{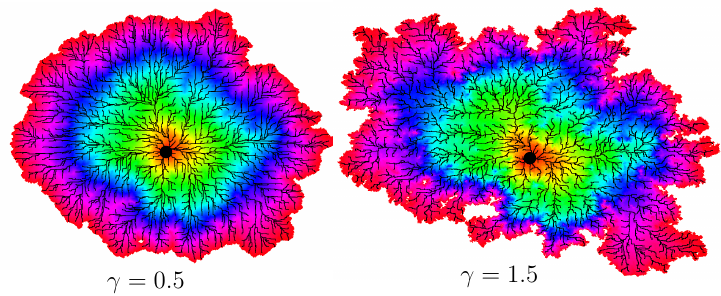}   
\includegraphics[width=0.3\textwidth]{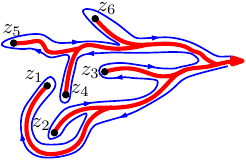}   
\caption{\label{fig-ball-sim} 
\textbf{Left and middle:} Simulations of LQG metric balls with two different values of $\gamma$ and the same GFF instance, produced using code provided by J.\ Miller. Colors indicate the distance to the center point (black dot). The black curves are LQG geodesics from the center point to other points in the ball. Notice that these geodesics have a tree-like structure. 
\textbf{Right:} Points $z_1,\dots,z_6 \in \BB C\setminus Z_\infty$ which admit unique $D$-geodesics to $\infty$ (red). The points $z_1,\dots,z_6$ are each hit once by the space-filling path $g$ of Theorem~\ref{thm-order}, in order. Points on the geodesics are hit at least twice. The blue path gives a rough indication of the order in which $g$ visits points of $\BB C$. 
}
\end{center}
\end{figure}

\begin{comment}
xi[gamma_] := gamma/(2 + gamma^2/2 + gamma/Sqrt[6])
Solve[xi[gamma] == 0.2, gamma]
Solve[xi[gamma] == 0.4, gamma]
\end{comment}

Our main theorem concerning the LQG metric gives the existence of a CaTherine wheel one of whose zippers is the tree of LQG geodesics to $\infty$. To state this precisely, we recall some background on LQG geodesics. 
Let $w\in\BB C$. A {\em $D$-geodesic from $w$ to $\infty$} is a path $\sigma : [0,\infty) \to\BB C$ such that $\sigma(0)=w$ and $D(\sigma(t), \sigma(s)) = s-t$ for each $0\leq t < s <\infty$. The following is a restatement of~\cite[Proposition 4.4]{lqg-zero-one}.

\begin{proposition} \label{prop-geo-infty}
Almost surely, for each $z\in\BB C$ there exists a $D$-geodesic from $z$ to $\infty$. For each fixed $z\in\BB C$, a.s.\ this $D$ geodesic is unique (but there can be exceptional points for which is it not unique). 
Furthermore, for each $r > 0$, there exists an $R > r$ with the following confluence property. If $\sigma$ and $\sigma'$ are $D$-geodesics from points in the open Euclidean disk $B_r(0)$ to $\infty$, then $\sigma \settminus B_R(0)= \sigma' \settminus B_R(0)$.
\end{proposition}

We define the {\em $D$-geodesic tree rooted at $\infty$}, denoted $Z_\infty$, to be the union of all of the $D$-geodesics to $\infty$, minus their starting points, i.e., 
\eqb   \label{eqn-geo-tree-infty} 
Z_\infty := \left\{ \sigma(t) \,:\, \text{$\sigma$ is a $D$-geodesic from a point of $\BB C$ to $\infty$ and $0<t<\infty$} \right\} .
\eqe 
Proposition~\ref{prop-geo-infty} immediately implies that $Z_\infty$ is dense in $\BB C$ and is path-connected.

\begin{theorem}[CaTherine wheel for LQG] \label{thm-lqg-wheel}
Let $\gamma\in (0,2)$ and let $D$ be the $\gamma$-LQG metric as above.
View $\BB C$ as a subset of the Riemann sphere $\BB C\cup\{\infty\}$.
There exists a unique CaTherine wheel $f : S^1 \to \BB C\cup\{\infty\}$ whose right zipper satisfies $Z^+= Z_\infty$.
\end{theorem}
 
The CaTherine wheel $f$ can be thought of as the contour exploration of the LQG geodesic tree. The following theorem makes this precise and also establishes some other properties of $f$ which are reminiscent of properties of random space-filling curves seen elsewhere in probability. See Figure~\ref{fig-ball-sim}, right, for an illustration.

\begin{theorem}[Properties the LQG CaTherine wheel] \label{thm-order}
Let $f$ be as in Theorem~\ref{thm-lqg-wheel}.
\begin{enumerate}
\item 
Almost surely, $f^{-1}(\infty)$ consists of a single point, and there exists a continuous space-filling curve $g : \BB R\to \BB C$ which is a re-parametrization of $f|_{S^1\settminus f^{-1}(\infty)}$ such that $g(0) = 0$ and $\mu(g[a,b]) = b-a$ for each $a,b\in\BB R$ with $a < b$, where $\mu$ is the LQG area measure. 
\item Let $z,w\in\BB C \setminus Z_\infty$ such that there is a unique $D$-geodesic from each of $z$ and $w$ to $\infty$ (by Proposition~\ref{prop-geo-infty}, this is a.s.\ the case for a fixed $z,w\in\BB C$). Then with $g$ as above, each of $z$ and $w$ has a unique preimage under $g$. Moreover, $z$ is hit before $w$ by $g$ if and only if the $D$-geodesic from $z$ to $\infty$ merges into the $D$-geodesic from $w$ to $\infty$ from the right side of the latter. 
%\item Each point of $Z_\infty$ is visited at least twice by $g$. 
\end{enumerate}
\end{theorem}

We will prove Theorem~\ref{thm-lqg-wheel} in Section~\ref{sec-lqg-proof} by checking the conditions of Theorem~\ref{theorem:half_zipper}. To do this, we will cite some results from the existing LQG literature (mostly from~\cite{lqg-zero-one}) which give certain geometric properties for LQG geodesics. Then, we will deduce the conditions of Theorem~\ref{theorem:half_zipper} from these properties using purely topological arguments. Our proof can be read without any knowledge of LQG if one is willing to take some external results as black boxes. Theorem~\ref{thm-order} will follow easily from some features of the construction of $f$ (see in particular Lemmas~\ref{lemma:unique_preimage} and~\ref{lemma:circular_order}). 

We state and prove Theorems~\ref{thm-lqg-wheel} and~\ref{thm-order} in the setting of the whole-plane GFF for convenience, but analogous statements should be true for other variants of the GFF, e.g., the variant corresponding to the unit-area LQG sphere (see~\cite[Section 4.5]{wedges} or~\cite{dkrv-lqg-sphere,ahs-sphere}). We expect that such variants can be deduced from Theorem~\ref{thm-lqg-wheel} and the intermediate results in its proof via local absolute continuity arguments. We expect that Theorem~\ref{thm-lqg-wheel} is also true in the critical case $\gamma=2$, but we will not treat this case here since several of the results we need to cite have only been proven for $\gamma \in (0,2)$. 

Theorem~\ref{thm-lqg-wheel} is an LQG metric analog of other theorems in the literature which construct a peano curve from one or more trees in the plane. Examples of such theorems include the construction of space-filling SLE$_\kappa$ for $\kappa\in (4,8)$ in~\cite{ig4} and the construction of a peano curve for the geodesic tree in the directed landscape~\cite{bhatia-peano-curve}. It may be possible to use Theorem~\ref{theorem:half_zipper} to construct CaTherine wheels (equivalently, peano curves) from geodesic trees in other random geometries, e.g., the Poisson roads metric~\cite{bck-poisson-roads,kendall-poisson-roads} or the hypothetical directed versions of the LQG metric~\cite[Conjecture 2.2]{bg-bipolar-directed}.  

Theorem~\ref{thm-lqg-wheel} is new when $\gamma \in (0,2)\settminus\{\sqrt{8/3}\}$. But, in the special case when $\gamma=\sqrt{8/3}$, the CaTherine wheel corresponding to the LQG geodesic tree is already well-studied, at least in the unit-area LQG sphere case. This is because $\sqrt{8/3}$-LQG is equivalent to the Brownian map~\cite{lqg-tbm1,lqg-tbm2}; see Section~\ref{sec-wheel-prob} for details. 

%Indeed, it is shown in~\cite{lqg-tbm1,lqg-tbm2} that $\BB C\cup\{\infty\}$, equipped with the $\sqrt{8/3}$-LQG metric induced by the LQG sphere field, is isometric (as a metric space) to the {\em Brownian map} of~\cite{legall-uniqueness,miermont-brownian-map}. In this case $f$, when parametrized by LQG area, is the same as the quotient map $S^1 \to \BB C\cup\{\infty\}$ arising from the usual Schaeffer-style construction of the Brownian map (see~\cite{legall-uniqueness,miermont-brownian-map}). This perspective plays a fundamental role in~\cite{lqg-tbm3}. In this setting, the zippers corresponding to $f$ admit an explicit description in terms of the Brownian snake. We do not have any reason to believe that such a nice description exists for general $\gamma\in (0,2)$. 

\subsection{CaTherine wheels in probability theory} 
\label{sec-wheel-prob}

As noted in Section~\ref{sec-wheel}, CaTherine wheels arise frequently in probability theory, although the term ``CaTherine wheel'' is not used in the probability literature. Instead, CaTherine wheels are usually referred to as peano curves (satisfying certain conditions). Zippers in probability theory are usually referred to instead as pairs of trees. To put our results in context, we will briefly survey the appearances of CaTherine wheels in probability. 

Perhaps the most well-studied CaTherine wheels in probability theory are {\em Schramm-Loewner evolution} (SLE$_\kappa$) curves for $\kappa \geq 8$~\cite{schramm0}. Specifically, whole-plane SLE$_\kappa$ from $\infty$ to $\infty$ is a CaTherine wheel (when viewed as a curve in the Riemann sphere) for $\kappa \geq 8$. The associated zipper consists of the trees of SLE$_{16/\kappa}(2- 16/\kappa)$ flow lines of a whole-plane Gaussian free field; see~\cite[Section 1.2.3]{ig4} and~\cite[Footnote 4]{wedges}. Space-filling SLE$_\kappa$ for $\kappa\in (4,8)$ is not exactly a CaTherine wheel since the image of an interval under such a curve is not homeomorphic to a closed disk (rather, it has cut points). But, space-filling SLE$_\kappa$ for $\kappa \in (4,8)$ has many similar features as CaTherine wheels, including an associated zipper-like pair of trees (c.f.\ Remark~\ref{remark-weaker} and~\cite[Section 5]{Catherine_wheels}). The seminal {\em mating of trees} theorem of Duplantier-Miller-Sheffield~\cite{wedges} gives a description of the zipper associated with space-filling SLE$_\kappa$ for $\kappa > 4$ in terms of a pair of correlated Brownian motions. 

Another important random CaTherine wheel is the peano curve on the {\em Brownian map}. The Brownian map is a random metric space $(X,D)$ homeomorphic to $S^2$ which arises as the Gromov-Hausdorff scaling limit of uniform triangulations, quadrangulations, etc.~\cite{legall-uniqueness,miermont-brownian-map}. It can be constructed as a metric quotient of $S^1$ under a certain equivalence relation defined in terms of the {\em Brownian snake}. The quotient map $f : S^1\to X$ is a CaTherine wheel. The associated zipper has a simple description in terms of the Brownian snake. In particular, one of the half-zippers is the geodesic tree on the Brownian map (rooted at a uniformly sampled marked point) and the other has the law of the continuum random tree. See~\cite[Section 4]{tbm-characterization} for explanations. It is shown in~\cite{lqg-tbm1,lqg-tbm2} that $\BB C\cup\{\infty\}$, equipped with the $\sqrt{8/3}$-LQG metric induced by the LQG sphere field, is isometric (as a metric space) to the Brownian map. Hence the CaTherine wheel associated with the Brownian map is the same as the CaTherine wheel in Theorem~\ref{thm-lqg-wheel} for $\gamma=\sqrt{8/3}$, with the LQG sphere field in place of the whole-plane GFF. We do not have any reason to expect that the zippers in Theorem~\ref{thm-lqg-wheel} for general $\gamma\in (0,2)$ admit a simple description analogous to the Brownian snake description in the Brownian map setting. 

Yet another CaTherine wheel in probability is the peano curve for the geodesic tree in the {\em directed landscape}~\cite{bhatia-peano-curve}. The construction of this curve is inspired by the construction of the Toth-Werner curve on the Brownian web~\cite{tw-brownian-web-peano}, which is a space-filling curve built from a pair of trees but is not quite a CaTherine wheel. 

In all of the above settings, both halves of the zipper associated with the random CaTherine wheel (or the space-filling curve which is not quite a CaTherine wheel) are well-understood. An important novel feature of this paper is that we only have a priori information about one half of the associated zipper.

\section{CaTherine wheels from half-zippers}
\label{section:half_zippers}

In this section we will prove Theorem~\ref{theorem:half_zipper} and Proposition~\ref{proposition:wheels_to_short_hair} using results from~\cite{Catherine_wheels} and elementary topological arguments.  

\begin{remark}
It is a fact that if $X$ is a topological space homeomorphic to a finite graph, any
embedding of $X$ in $S^2$ is tame; i.e.\/ there is a homeomorphism of $S^2$ to
itself taking the image of $X$ to a finite polygonal graph. See e.g.\/ \cite{Newman}. 
Thus we may reason about compact paths in $Z$ (and, if $Z$ has short hair, 
noncompact paths) combinatorially.
\end{remark}

\subsection{Half-zippers}

We start off by establishing that a half-zipper with short hair (Definition~\ref{definition:half_zipper}) automatically satisfies several other useful properties. 

\begin{lemma}[Increasing union]\label{lemma:increasing_union}
Suppose $Z$ is a half-zipper with short hair. Then we may write $Z$ as an
increasing union $Z:=\bigcup_n T_n$ where each $T_n$ is a finite tree, and such
that every component of $Z-T_n$ has diameter $<1/n$.
\end{lemma}
\begin{proof}
Since every point of $Z$ is a cut point, it contains no embedded loops. It follows
that the union of any finite set of finite subtrees of $Z$ is a finite subtree.
Let $T_n'$ be a sequence of finite subtrees of $Z$ as in the definition of
the short hair property with $\epsilon = 1/n$ and let $T_n = \cup_{i\le n}T_n'$.
Then $T_n$ is an increasing family. Let $T_\infty = \cup_n T_n$.
Then by construction every component of $Z-T_\infty$ has diameter $0$ and therefore
consists of at most one point. But any point of $Z-T_\infty$ could not be a cut
point of $Z$, and therefore $Z=T_\infty$ as claimed.
\end{proof}
From now on, if $Z$ is a half-zipper with short hair, we fix a sequence $T_n$ of
finite subtrees as in Lemma~\ref{lemma:increasing_union}.

\begin{definition}[Landing ray]\label{definition:landing}
A {\em ray} is a half-open interval in $Z$ (i.e.\/ a subspace homeomorphic to $[0,1)$).
It is a {\em proper ray} if the inclusion $[0,1) \subset Z$ is a proper embedding. 
A ray {\em lands} if it extends continuously to an embedding of $[0,1]$; the
image of $1$ is called the {\em landing point}.

A half-zipper has the {\em strong landing property} if every ray in $Z$ lands, and
if every point in $S^2$ is the landing point of some ray in $Z$.
\end{definition}

\begin{lemma}[Limit rays]\label{lemma:limit_rays}
Suppose $Z$ is a half-zipper with short hair. Fix $q\in Z$ and let $p_i \in Z$
be any sequence of points, and let $\eta_i\subset Z$ be the unique embedded path 
in $Z$ from $q$ to $p_i$.  Suppose $p_i \to p \in S^2$. 
Then after passing to a subsequence if necessary we may arrange that
\begin{enumerate}
\item{for any fixed $n$ the intersections $\eta_i \cap T_n$ converge as
subsets to a limit $\sigma_n \subset T_n$;}
\item{the $\sigma_n$ are an increasing nested sequence of (possibly degenerate) paths; and}
\item{the union $\sigma:=\bigcup_n \sigma_n$ is either a ray which lands at $p$ or
the unique path in $Z$ from $q$ to $p$ (which can only happen if $p\in Z$).}
\end{enumerate}
Consequently, any half-zipper with short hair has the strong landing property.
\end{lemma}
\begin{proof}
Since each $T_n$ is finite it is compact, and therefore $\eta_i \cap T_n$ has
a convergent subsequence for each fixed $n$. Pass to a diagonal subsequence of the
indices $i$, so that $\eta_i \cap T_n$ converges for each fixed $n$ to some
$\sigma_n$. Each $\sigma_n$ is a possibly degenerate path (i.e.\/ for some $n$ 
it might consist only of the initial point $q$), and since 
$T_n \subset T_{n+1}$ we necessarily have 
$\sigma_n \subset \sigma_{n+1}$ as an initial segment, so these paths form a nested
family. Since the diameter of $\sigma - \sigma_n$ is at most $1/n$ for any $n$ 
either $\sigma$ is a path that ends at $p$ or a ray that lands at $p$.
Note that if $p\in S^2-Z$ then the union $\sigma$ is necessarily a (proper) ray.

Now let us show that $Z$ has the strong landing property. First of all, because
$Z$ is dense, we have already constructed a landing ray that lands at any
point $p\in S^2 - Z$. Thus it remains to check that every ray lands. Let $\sigma$ be
a proper ray. Then if we define $\sigma_n = \sigma \cap T_n$ we have $\sigma = \bigcup_n \sigma_n$
and the diameter of $\sigma - \sigma_n$ is at most $1/n$, and thus $\sigma$ 
necessarily lands and we are done.
\end{proof}

\begin{definition}[Hairy]\label{definition:hairy}
A half-zipper is {\em hairy} if every embedded path $I\subset Z$ has branching
on both sides; i.e., there are nontrivial oriented paths $\alpha,\beta \subset Z$
that begin on $I$ and intersect $I$ only at these initial points, and which lie 
respectively to the right and to the left of the oriented interval $I$ in $S^2$. 
\end{definition}

\begin{lemma}[Short hair hairy]\label{lemma:short_hair_hairy}
Suppose $Z$ is a half-zipper with short hair. Then $Z$ is hairy.
\end{lemma}
\begin{proof}
Suppose not, so that there is some interval $I\subset Z$ with no branching on the right
for some choice of orientation. Let $p_i \in Z$ be a sequence of points converging
to the midpoint $p$ of $I$ from the right. Let $U$ be an open disk in $S^2$ centered
at $p$ and bisected (topologically) by $I$ 
(i.e.\/ intersecting $I$ in a single proper subinterval),
and let $V$ be the open half-disk component of $U$ in $S^2-I$ on the right of $I$.
Then $p_i$ is contained in $V$ for all sufficiently large $i$.

By Lemma~\ref{lemma:limit_rays} there is a
proper ray $\sigma$ in $Z$ starting at $p$ which is a limit of the unique embedded
paths from $p$ to $p_i$. Let $T_n$ be as in Lemma~\ref{lemma:increasing_union}. For fixed $n$ and big
$i$ the endpoint of $\sigma_n:=\sigma \cap T_n$ 
is joined by a path in $Z$ of diameter at most $1/n$ to
$p_i$ and is therefore (because $I$ has no branching on the right)
not on $I$, but lies in $V$. In particular the image of
$\sigma$ is a nontrivial closed embedded loop in $Z$, which is a contradiction.  
\end{proof}

\subsection{Universal circle}
\label{sec-universal-circle}

In this subsection we show that associated to any hairy half-zipper $Z$ is a
canonical circle $S^1_Z$ which completes a natural circular order on the set
of (topological) ends of $Z$.

We follow the construction and the notation in \cite{Catherine_wheels}, \S~3 
and \cite{Zippers}, \S~2.5 closely, and refer to those papers for details
and proofs.

Let $\EE$ denote the set of equivalence
classes of proper (unparameterized) rays, where two proper rays are 
equivalent if they agree outside a compact subset. We remark that as a set, this is
the same as the set of topological ends of $Z$ in its path-topology in the
sense of Freudenthal \cite{Freudenthal}, i.e.\/ the inverse limit of $\pi_0(Z-T_n)$ for any
cofinal sequence of compact path-connected subsets, for instance the set
of trees $T_n$ as in Lemma~\ref{lemma:limit_rays}; however we will shortly give
it a different topology than Freudenthal's end topology.

\begin{definition}[Convex hull]\label{definition:convex_hull}
Let $S$ be a subset of $Z\sqcup \EE$. If $S$ is empty or consists of a
single point of $\EE$, the convex hull of $S$ is empty. Otherwise the
convex hull is the smallest path-connected subset of $Z$ containing $S\cap Z$
and containing rays representing every element of $S\cap \EE$.
\end{definition}

The planar embedding and the orientation of $S^2$ defines a circular order on $\EE$,
by \cite{Zippers} Lemma~2.13.
We topologize $\EE$ with the order topology (this is quite different from the
Freudenthal topology). Since $Z$ is a countable union of finite
trees, $\EE$ is separable. Let $\overline{\EE}$ define the order completion. This is
circularly ordered, and is compact and separable in the order topology. The points
of $\overline{\EE} -\EE$ are represented by geometric objects called {\em ideal gaps}.
These are defined in \cite{Zippers}, \S~2.5 and we repeat the definition here. 

If $e^L,e^R$ are distinct elements of $\EE$ we denote by $[e^L,e^R]$ 
the set of ends $e\in \EE$ for which
$e^L,e,e^R$ is positively (i.e.\/ anticlockwise) circularly ordered, together with
$e^L, e^R$ themselves.

\begin{lemma}[Nonempty intersection]\label{lemma:nonempty_intersection}
Suppose there is an infinite nested sequence
$$\cdots [e_{i+1}^L,e_{i+1}^R] \subset [e_i^L,e_i^R] \subset\cdots$$
whose intersection is empty. For each $i$ let $\eta_i \subset Z$ denote the
convex hull of $e_i^L$ and $e_i^R$ (this is a properly embedded line), oriented
to run from $e_i^L$ to $e_i^R$. Then there is some $i$ so that the intersection
$X: = \cap_{j\ge i}\eta_j$ is either a point, or a compact, nonempty interval.
\end{lemma}
This is \cite{Zippers}, Lemma~2.15.

\begin{definition}[Ideal gap]
A sequence of proper lines $\eta_i$ in $Z$ which are the convex hulls of 
$\lbrace e_i^L,e_i^R\rbrace$ for a sequence of nested intervals 
$[e_i^L,e_i^R]$ in $\EE$ is said to be {\em end nested}.

A point $p\in Z$ together with an end nested sequence $\eta_i$ as above for which
$\bigcap_i [e_i^L,e_i^R]$ is empty and for which $p \in \bigcap_i \eta_i$ is
called an {\em ideal gap}. We define an equivalence relation on ideal
gaps as follows. Let $S$ and $S'$ be a pair of sequences of nested intervals
in $\EE$ with empty intersection. We say $S$ refines $S'$ if $S'$ is a subsequence
of $S$. The relation of refinement is not an equivalence relation,
but it generates an equivalence relation on sequences of nested intervals
in $\EE$ with empty intersection by
$S \sim S'$ if there is a finite sequence $S=S_0,S_1,\cdots,S_n=S'$ 
so that for each $i$ either $S_i$ refines $S_{i+1}$ or $S_{i+1}$ refines $S_i$.
Say that two ideal gaps are equivalent if their associated
sequences of nested intervals are related by this equivalence relation.

A point $p$ as above is said to lie in the {\em support} of the 
equivalence class of the given ideal gap.
\end{definition}

The concept of an equivalence class of ideal gap is very similar to the concept of a prime end (see, 
e.g.,~\cite[Section 2]{pom-book} for the definition of a prime end); morally
speaking an ideal gap would be a prime end associated to the ``complementary region''
$S^2 - Z$ if this were an open disk. One could make this analogy more precise by
working directly with prime ends of the domains $S^2 - T_n$ and passing to a limit,
but there seems to be no particular advantage to this. 

\begin{lemma}[Order completion]\label{lemma:order_completion}
The union of $\EE$ together with the equivalence classes of ideal gaps of $Z$ admit
a natural circular ordering which is equal to the order completion $\overline{\EE}$.
\end{lemma}
This is \cite{Catherine_wheels}, Lemma~3.8.

\begin{lemma}[Ideal gaps are points]\label{lemma:single_point}
Suppose $Z$ is hairy. Then the support of each equivalence class of ideal gap is
a single point.
\end{lemma}
This is \cite{Catherine_wheels}, Lemma~3.11.

\begin{lemma}[Order completions are circles]\label{lemma:order_completion_circle}
Suppose $Z$ is hairy. Then the order completion $\overline{\EE}$ is homeomorphic to
a circle $S^1_Z$.
\end{lemma}
This is \cite{Catherine_wheels}, Lemma~3.12.

\medskip

Since every point of $Z$ is a cut point, every point $p\in Z$
is in the support of an ideal gap. Although it is not logically necessary, we
explain this briefly for the benefit of the reader. For any $p$ the
complement $Z - p$ has at least two components, and each of these components
contains many proper rays. If $A$ is a component of $Z - p$ and $\EE_A$ is the
set of ends corresponding to proper rays in $A$, then $\EE_A$ is dense in some
interval of $S^1_Z$ and for different components $A,A'$ the subsets
$\EE_A$ and $\EE_{A'}$ are disjoint and pairwise unlinked in $S^1_Z$. Thus 
there is an induced circular order on the set of components of $Z-p$. An ideal
gap supported at $p$ is a (Dedekind) cut in this circularly ordered set. Since
there are at least two components of $Z-p$, there are at least two cuts in the
circularly ordered set, and therefore $p$ is contained in the support of at least two
distinct ideal gaps (and, by Lemma~\ref{lemma:single_point}, it is {\em equal} to
the support of at least two distinct ideal gaps). 

Let's introduce some terminology.
If $\eta$ is an oriented path in $Z$ and $p\in \eta$, let $A^+$, resp.\ $A^-$, 
be the components of $Z-p$ containing the part of $\eta$ that is ahead
of,  resp.\ behind, $p$ in the orientation. Say that an ideal gap $e$ 
supported at $p$ lies {\em on the positive side} of $\eta$
(equivalently the right hand side) if the corresponding cut of the components
of $Z-p$ lies in the (oriented) interval $[A-,A+]$ in the circular order on
components of $Z-p$.

\subsection{CaTherine wheels}\label{subsection:catherine_wheels}

\begin{comment}

\begin{definition}[CaTherine wheel]\label{definition:wheel}
A {\em CaTherine wheel} is a continuous surjective map $f:S^1 \to S^2$ such that
for every closed interval $J\subset S^1$ the image $f(J)$ is homeomorphic to a closed
disk $D^2$, and $f(\partial J)\subset \partial f(J)$.
\end{definition} 

\end{comment}

%The proof occupies the remainder of this subsection.
By Lemma~\ref{lemma:short_hair_hairy} $Z$ is hairy, and by
Lemma~\ref{lemma:order_completion_circle} the order completion $\overline{\EE}$
is homeomorphic to a circle $S^1_Z$. By Lemma~\ref{lemma:order_completion}
the points of $S^1_Z$ correspond to ends of $Z$, and to ideal gaps. Define
$f:S^1_Z \to S^2$ as follows: if $e\in S^1_Z$ is an end, define $f(e)$ to be 
the landing point of any ray representing $e$; such a landing point exists, because
$Z$ has the strong landing property, by Lemma~\ref{lemma:limit_rays}. 
If $e\in S^1_Z$ is an ideal gap, define $f(e)$ to be the support of $e$. 
This is a single point by Lemma~\ref{lemma:single_point}. Thus $f$ is well-defined, and,
by the strong landing property, it is surjective.
We claim that $f$ is a CaTherine wheel. This is established in a sequence of steps.

\begin{lemma}[Continuous]\label{lemma:f_continuous}
The map $f:S^1_Z \to S^2$ is continuous.
\end{lemma}
\begin{proof}
Let $n\in\mathbb N$ and let $T_n$ be as in Lemma~\ref{lemma:increasing_union}.
If $e\in S^1_Z$ is an end, and $\sigma$ is a ray
representing $e$, there is a point $p$ on $\sigma$ in the complement of $T_n$
for any large $n$. The component of $Z-p$ containing $e$ therefore has diameter
$<1/n$, and pairs of rays in this end branching from either side of $\sigma$ give
intervals $[e_i^L,e_i^R]$ nesting down to $e$. In particular, if $U$ is
any open neighborhood of $f(e)$, then $f^{-1}(U)$ contains the entire open 
interval $(e_i^L,e_i^R)$ about $e$ for all sufficiently large $i$.
If $e\in S^1_Z$ is an ideal gap with support $p$, we may similarly find a 
nested sequence of intervals $[e_i^L,e_i^R]$ for which the
convex hulls $\eta_i$ of $\lbrace e_i^L,e_i^R\rbrace$ have intersection $p$.
Again, fix a large $n$ and consider the intersections of $\eta_i$ with $T_n$.
For sufficiently large $i$ these intersections must have arbitrarily small diameter, or else
$\bigcap_i \eta_i$ would contain an interval in $T_n$. But the diameter of
each component of $\eta_i - T_n$ is less than $1/n$. It follows once more
that if $U$ is any open neighborhood of $f(e)$, then $f^{-1}(U)$ contains
the entire open interval $(e_i^L,e_i^R)$ about $e$ for all sufficiently large $i$. This
proves continuity of $f$.
\end{proof}

To show that $f$ is a CaTherine wheel, it remains to prove the following.

\begin{lemma}[Disks]\label{lemma:disks}
Let $J \subset S^1_Z$ be a closed interval.
Then $f(J)$ is a closed topological disk, and $f(\partial J)$ is contained in 
$\partial f(J)$.
\end{lemma}

The proof of Lemma~\ref{lemma:disks} requires several lemmas. To start off, if $J$ is a closed interval in $S^1_Z$ then the orientation on $S^1_Z$
determines an orientation on $J$. With respect to this orientation we
denote the endpoints of $J$ as $J^\pm$ so that $J = [J^-,J^+]$.
Now for every such closed interval $J$ in $S^1$ let $Z_J$ 
be the convex hull in $Z$ of the 
corresponding set of ends and supports of ideal gaps in $J$, and let
$\eta_J \subset Z_J$ be the convex hull of the ends and supports of ideal
gaps in $\partial J$, oriented to run from $f(J^-)$ to $f(J^+)$. 

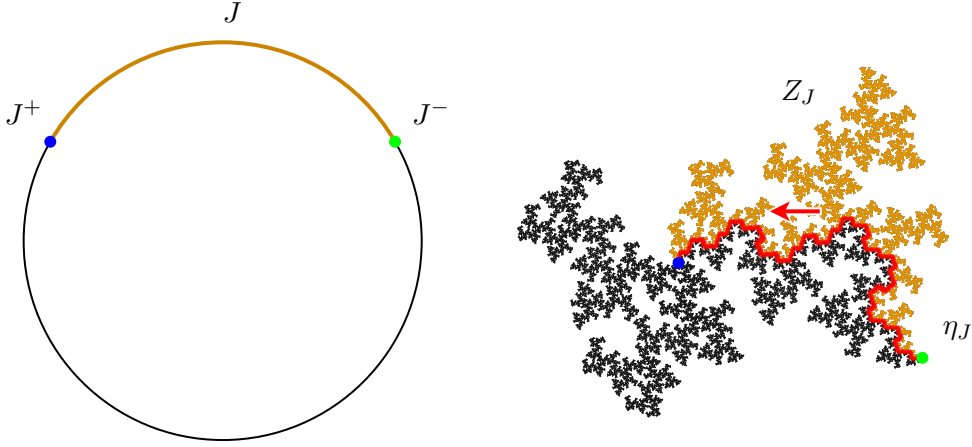
\begin{figure}[htpb]
\labellist
\small\hair 2pt
\pinlabel $J$ at 135 240
\pinlabel $J^+$ at 30 190
\pinlabel $J^-$ at 235 190
\pinlabel $\eta_J$ at 500 80
\pinlabel $Z_J$ at 420 200
\endlabellist
\centering
\includegraphics[scale=0.75]{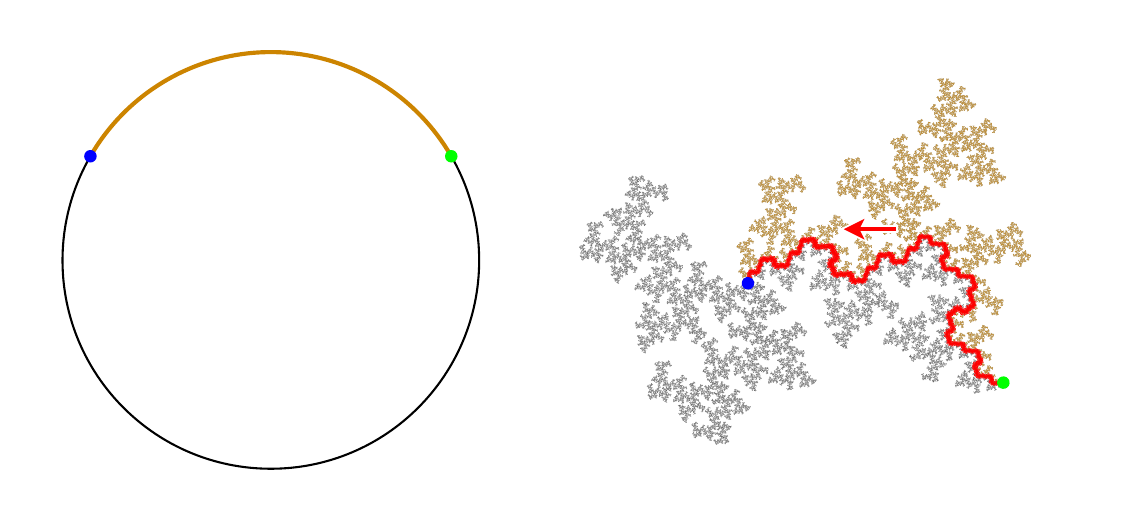}
\caption{\label{fig-J_Z_J_eta_J} $J\subset S^1$ runs from $J^-$ to $J^+$ in $S^1$. The convex hull 
in $Z$ of $J^\pm$ is $\eta_J$ (in red). The endpoint $J^+$ (in blue) is an ideal gap
supported at a point in $Z$ whereas $J^-$ (in green) is an end. Thus $\eta_J$
is a half-open interval in $Z$ that lands at $f(J^-)$, and
separates $Z_J - \eta_J$ (in orange) from $Z - Z_J$ in $Z$. However, the orientation
on $\eta_J$ (indicated by the arrow) runs from $f(J^-)$ to $f(J^+)$.}
\end{figure}

\begin{lemma}[Separation]\label{lemma:separation}
The path $\eta_J$ separates $Z-Z_J$ from $Z_J-\eta_J$ in $Z$.
\end{lemma}
\begin{proof}
If $p\in Z-Z_J$ then any path in $Z$ from $p$ to $\eta_J$ lies on the left
(i.e.\/ the negative) side of $\eta_J$ whereas if $p \in Z_J - \eta_J$
any path in $Z$ from $p$ to $\eta_J$ lies on the right (i.e.\/ the positive) side
of $\eta_J$ by definition. Thus if $p \in Z-Z_J$ and $q\in Z_J - \eta_J$
the unique path in $Z$ from $p$ to $q$ must intersect $\eta_J$. See Figure~\ref{fig-J_Z_J_eta_J}.
\end{proof}

\begin{lemma}[Closure of $Z_J$]\label{lemma:closure}
The closure $\overline{Z}_J$ is equal to $f(J)$, and $\overline{Z}_J\cap Z = Z_J$.
\end{lemma}
\begin{proof}
Evidently $f(J)$ is closed. It contains $Z_J$ because if 
$p,q \in J$ have convex hull $\eta$ in $Z$, and we orient $\eta$ so that
it runs from $p$ to $q$ where the oriented interval $[p,q]\subset S^1_Z$ is
contained in $J$, then every ideal gap supported at any $r\in \eta$ on
the positive side of $\eta$ is in $J$ and therefore $r \in f(J)$.
Hence $f(J)$ contains $\overline{Z}_J$.
Conversely, suppose $r$ is in the interior of $J$. Then by the
definition of the circular order, the convex hull of $f(J^-)$ and $f(r)$ lies on the
right of $\eta_J$ so that either $f(r)$ is in $Z_J$ or $f(r)$ is a landing point of 
a ray in $Z_J$ and therefore in either case $f(r)\in \overline{Z}_J$.  Hence
$f(J)$ is contained in $\overline{Z}_J$.

But now $\overline{Z}_J \cap Z = f(J)\cap Z = Z_J$.
\end{proof}

Next we must discuss components of $S^2 - \overline{Z}_J$. We establish
some properties.

\begin{lemma}[Boundary points]\label{lemma:boundary_points}
Let $U$ be a component of $S^2 - \overline{Z}_J$. Then 
$Z_J \cap \overline{U}$ is nonempty, and is contained in
$\eta_J$. 
\end{lemma}
\begin{proof}
Pick $q\in U \cap Z$.  Since $Z$ is path-connected, there is
a path in $Z$ from $q$ to $Z_J$ and this path must intersect $\overline{U}$
in a point of $Z_J\cap \overline{U}$ because $\overline{Z}_J\cap Z = Z_J$
by Lemma~\ref{lemma:closure}.

Pick $q \in U \cap Z$ and a sequence $p_i \in U\cap Z$ converging to
any $p\in Z_J \cap \overline{U}$. By Lemma~\ref{lemma:limit_rays}
the paths $\sigma_i$ from $q$ to $p_i$ converge after passing to a 
subsequence to $\sigma$ that lands at $p$. By
Lemma~\ref{lemma:separation} no $\sigma_i$ can meet $Z_J$ except
in a subset of $\eta_J$ and therefore $\sigma$ meets $Z_J$ only in a
subset of $\eta_J$. In particular, $p\in \eta_J$.
\end{proof}

Conversely to Lemma~\ref{lemma:boundary_points} we have the following:
\begin{lemma}[$\eta_J$ in boundary]\label{lemma:gamma_in_boundary}
Suppose $\eta_J$ consists of more than one point. Then there is a unique
connected component $U$ of $S^2 - \overline{Z}_J$ such that
$\eta_J \subset \overline{U}$ and $U$ lies on the negative side of
$\eta_J$.
\end{lemma}
\begin{proof}
We first observe that if there is such a component $U$ then it is unique, since
if $p$ is an interior point of $\eta_J$ and $p_i \in S^2 - \eta_J$ converges to $p$ from
the left, $p_i$ must eventually be contained in $U$. 

The proof that such a component $U$ exists
is similar to that of Lemma~\ref{lemma:short_hair_hairy}.
Let $p\in \eta_J$ be an interior point, let $W\subset S^2$
be a small open disk centered at $p$ and bisected by $\eta_J$ and let
$V$ be the open half-disk component of $W$ in $S^2-\eta_J$ on the left
(i.e.\/ the negative) side of $\eta_J$.
We claim that if $W$ is small enough, 
$V$ is contained in some component of $S^2 - \overline{Z}_J$.

For otherwise, we can find $p_i \in Z_J\cap V$ which converge to $p$ from
the left. Let $\sigma_i$ be a sequence of paths in $Z_J$ from $p$ to $p_i$.
By Lemma~\ref{lemma:limit_rays} they converge to some $\sigma$. Since each $\sigma_i$
is a path in $Z_J$ it cannot exit $\eta_J$ to the left; in particular, if we choose
a closed interval $\tau$ in $S^2 - V$ so that $\eta_J \cup \tau$ is a Jordan arc, 
each $\sigma_i$ must intersect $\tau$ and therefore the same is true of $\sigma$.
Thus as in the proof of Lemma~\ref{lemma:short_hair_hairy} the image of $\sigma$ is
a nontrivial closed embedded loop in $Z_J$, which is a contradiction. 
\end{proof}

\begin{lemma}[No local cut points]\label{lemma:no_local_cut_points}
The set $\overline{Z}_J$ has no local cut points in $Z_J$. 
\end{lemma}
\begin{proof}
Suppose $p \in Z_J$ is a local cut point of $\overline{Z}_J$, so that there is a
neighborhood $U$ of $p$ for which there are at least two components $A$, $B$
of $U \cap (\overline{Z}_J - p)$. Without loss of generality we may assume that 
$A\cup p$ contains a path in $Z_J$ starting at $p$. We claim that the same is true of
$B$. For, otherwise, if we choose $p_i \in B\cap Z_J$ converging to $p$ and let
$\sigma_i$ be the paths in $Z_J$ from $p$ to $p_i$, then none of the $\sigma_i$ are
entirely contained in $B\cup p$, and therefore by Lemma~\ref{lemma:limit_rays} 
they converge to a nontrivial loop in $Z_J$, which is a contradiction. See
Figure~\ref{fig-local_cut_point}.

\begin{figure}[htpb]
\centering
\includegraphics[scale=1]{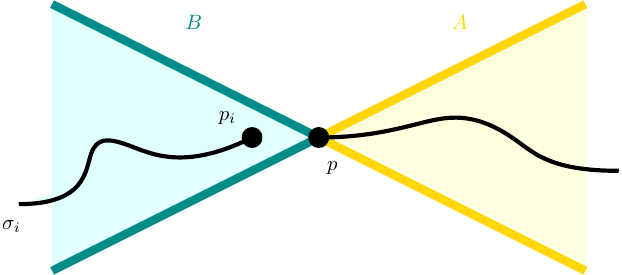}
\caption{\label{fig-local_cut_point} As $p_i \to p$ the paths $\sigma_i$ converge
to a nontrivial loop in $Z_J$.}
\end{figure}

We may therefore assume that there is a path $\eta \subset Z_J$ containing
$p$ and separating $\eta-p$ into $\eta^- \subset A$ and $\eta^+ \subset B$. 
Because $p$ is a local cut point of $\overline{Z}_J$ there is a sequence of
points $p_i \in Z - Z_J$ converging to $p$ from either side of $\eta$ and
by Lemma~\ref{lemma:limit_rays} we may therefore produce rays $r,r'$ in
$Z - Z_J$ landing on $\eta$ from either side. We shall show that
this gives a contradiction. 

In fact, any path
$\eta$ in $Z_J$ may be extended to a proper line $\eta'$ in $Z_J$ limiting to ends 
$e^\pm \in J$, and we may orient $\eta'$ so that it runs from $e^-$ to $e^+$
where the oriented interval $[e^-,e^+] \subset J$. For such an orientation, if
$\alpha$ is any ray in $Z-\eta'$ landing on $\eta'$ from the right 
(i.e.\/ positive) side, then any $p\in \alpha$ is
the support of an ideal gap in $[e^-,e^+]$, and therefore $\alpha \subset Z_J$.
Thus one of $r,r'$ must actually lie in $Z_J$ after all, and we obtain a contradiction.
\end{proof}

\begin{lemma}[One complementary component]\label{lemma:one_complement}
The complement $S^2 - \overline{Z}_J$ has exactly one component.
\end{lemma}
\begin{proof}
Suppose there are at least two complementary components $U$ and $V$.
By Lemma~\ref{lemma:boundary_points} both $\partial \overline{U}$ and
$\partial \overline{V}$ must intersect $\eta_J$, and if $\eta_J$ has
more than one point then by Lemma~\ref{lemma:gamma_in_boundary} we may
assume without loss of generality that $U$ contains all of $\eta_J$
and lies on the negative side of $\eta_J$. But then in either case
there must be some point $p\in \eta_J$ in $\partial \overline{U}\cap
\partial \overline{V}$. The set $\overline{Z}_J$ is locally connected, since it
is the continuous image of the interval $J$; in particular, the point $p$
is accessible from both $U$ and $V$, and therefore
there is a compact interval $I$ with one endpoint in each of $U$ and $V$
and intersecting $\overline{Z}_J$ only at $p$. Consequently $p \in Z_J$ is a
local cut point of $\overline{Z}_J$, contrary to Lemma~\ref{lemma:no_local_cut_points}.

To see there is at least one complementary component, suppose not so 
that $Z_J$ is dense in $S^2$. Choose $q\in Z_J$ and $q' \in Z - Z_J$.
There is a unique path $\tau$ in $Z$ joining $q$ to $q'$, and $\tau$
necessarily intersects $Z-Z_J$ in a half-open segment. Now choose
a sequence of points $p_i \in Z_J$ converging to $q'$. The paths
$\sigma_i$ from $q$ to $p_i$ all lie in $Z_J$, and converge by
Lemma~\ref{lemma:limit_rays} to $\sigma\subset Z_J$ landing at $q'$; thus
$\sigma \cup q'$ is a path in $Z$ from $q$ to $q'$ which is different
from $\tau$, which is a contradiction.
\end{proof}

\begin{proof}[Proof of Lemma~\ref{lemma:disks}]
We first show that $f(J)$ has the topology of a closed disk.
By Lemma~\ref{lemma:closure}, $f(J) = \overline{Z}_J$, with $Z_J$ as in the discussion just after Lemma~\ref{lemma:disks}. 
By Lemma~\ref{lemma:one_complement}, $S^2 - \overline{Z}_J$ has one component $U$. Since
$\overline{Z}_J = f(J)$, it is locally connected, so it suffices to show $\overline{Z}_J$ has no (global) cut points. But any
(global) cut point of $\overline{Z}_J$ must lie in $Z_J$ since $Z_J$ is
path-connected, and $\overline{Z}_J$ has no local cut points at all
that are in $Z_J$ by Lemma~\ref{lemma:no_local_cut_points}.

Finally, $f(\partial J)$ is contained in the closure of $\eta_J$, and
$\eta_J$ is contained in $\overline{U}\cap \overline{Z}_J =
\partial \overline{Z}_J$ by Lemma~\ref{lemma:boundary_points} and 
Lemma~\ref{lemma:gamma_in_boundary}.
\end{proof}

\begin{proof}[Proof of Theorem~\ref{theorem:half_zipper}]
The map $f:S^1_Z \to S^2$ that sends an ideal gap to its support and an end to the
landing point of the associated ray in $Z$ is continuous (Lemma~\ref{lemma:f_continuous})
and for every closed interval $J\subset S^1$ satisfies $f(J) = \overline{Z}_J$ 
(Lemma~\ref{lemma:closure}) where $Z_J \subset Z$ is the convex hull in $Z$ of $J$.
Furthermore, each $\overline{Z}_J$ is a disk, and $f(\partial J)$ is contained in
$\partial \overline{Z}_J$ (Lemma~\ref{lemma:disks}). Thus $f$ is a
CaTherine wheel, and by construction, $Z$ is $Z^+$ where
$Z^\pm$ is the zipper associated to $f$.
Conversely if $f:S^1 \to S^2$ is any CaTherine wheel with
zipper $Z^\pm$ then for any $J\subset S^1$ we can define $\eta^+_J \subset Z^+$ to be
the oriented convex hull of $J^\pm$, and let $Z^+_J$ be the subset of $Z^+$ on the positive
side of $\eta^+_J$ and then $f(J) = \overline{Z}^+_J$; see e.g.\/
\cite{Catherine_wheels} \S~1.5 especially Lemma~1.22. In particular, $Z^+$ (as a set) 
determines $f$, so there is a unique $f$ with $Z=Z^+$. This concludes the proof
of Theorem~\ref{theorem:half_zipper}. 
\end{proof}

We conclude this subsection with a proof of the converse to Theorem \ref{theorem:half_zipper}. 
%The notation and setup is taken directly from \cite{Catherine_wheels} \S~1 andthe proof of Theorem~1.22 therein, and we direct the reader to that paper for details.

\begin{proof}[Proof of Proposition~\ref{proposition:wheels_to_short_hair}]
Any CaTherine wheel $f$ extends canonically to $F:S^2_f \to S^2$ where $S^2_f = S^1 \sqcup P^\pm$
for two open hemispheres $P^\pm$, and $Z^\pm = F(P^\pm)$. 
Each of $P^\pm$ admits a metric for which it is isometric to an open
Euclidean unit disk, with $S^1$ as the unit circle, and so that for each $p \in Z^+$ (say) the
point preimage $F^{-1}(p) \cap P^+$ is a relatively closed contractible subset
bounded by Euclidean straight line segments and/or round circular arcs perpendicular to
the boundary circle. These point preimages form a monotone upper semi-continuous 
decomposition of $P^+$; see \cite{Catherine_wheels} \S~1.4 especially Theorem~1.20 
for all these assertions.
 
For any fixed metric on $S^2$, and for any $\epsilon>0$ there is
a $\delta>0$ so that every subset of $S^2_f$ of diameter $<\delta$ has image of
diameter $<\epsilon$. Let $X_\delta^+$ be the subset of $P^+$ consisting of the union
of all decomposition elements of diameter $\ge \delta$, together with those
for which some boundary arc is part of a circle of radius at least $\delta$. 
Then the image $F(X_\delta^+)$ is a finite tree, 
and the diameter of every component of $Z^+ - F(X_\delta^+)$ is $<\epsilon$, as desired.
See Figure~\ref{decomposition-tree}.

\begin{figure}[htpb]
\centering
\includegraphics[scale=0.5]{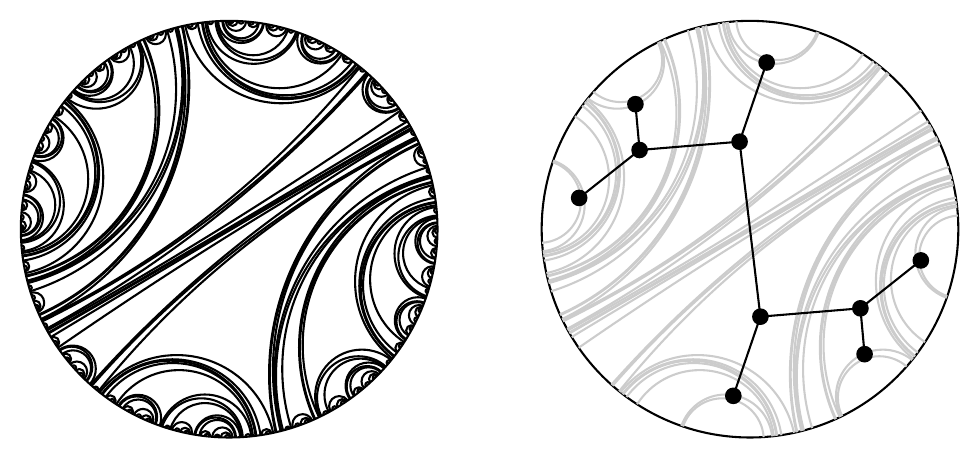}
\caption{The decomposition of $P^+$, the boundary arcs of radius $\ge \delta$, and the quotient tree.}
\label{decomposition-tree}
\end{figure}
\end{proof}

\subsection{Zippers}

We state and prove two lemmas about zippers that are used in \S~\ref{sec-lqg-proof}.
\begin{lemma}[Unique preimage]\label{lemma:unique_preimage}
Let $f:S^1 \to S^2$ be a CaTherine wheel with zipper $Z^\pm$. Let $p\in S^2 - Z^+$. Then
$f^{-1}(p)$ consists of a single point if and only if there is a unique equivalence
class of rays in $Z^+$ that lands at $p$.
\end{lemma}
\begin{proof}
If $f:S^1 \to S^2$ is a CaTherine wheel with zipper $Z^\pm$, then the set of
points $p \in S^2$ for which $f^{-1}(p)$ has more than one point is precisely
$Z^-\cup Z^+$; see \cite{Catherine_wheels} Theorem~1.23. Thus for $p \in S^2 - Z^+$
we must show that $p\in Z^-$ if any only if there are at least two equivalence classes
of rays in $Z^+$ that land at $p$.

If there are two inequivalent rays $r,r'$ in $Z^+$ that land at $p$ then (after replacing each of $r,r'$
by equivalent rays) the union $\eta:=r\cup r' \cup p$ is a Jordan curve with
$\eta - p \subset Z^+$. Since $Z^-$ is dense in $S^2$, it must intersect both
components of $S^2 - \eta$; since it is path-connected it must intersect $\eta$, and
since it is disjoint from $Z^+$, it follows that $p\in Z^-$.

Conversely suppose $p\in Z^-$. Then (as explained in the paragraph following
Lemma~\ref{lemma:order_completion_circle}) $p$ is in the support of at least two
ideal gaps in $S^1_-$, the circle which is the order completion of the space of
ends of $Z^-$. But we have seen in \S~\ref{subsection:catherine_wheels} 
that $S^1_-$ is canonically isomorphic to $S^1$ (the
domain of $f$) and also thereby to $S^1_+$, the circle which is the order completion of
the space of ends of $Z^+$. It follows that $p$ is associated to at least two distinct points
$e,e' \in S^1_+$. Neither of these can be ideal gaps in $S^1_+$ or else their support
(which is equal to $p$) would lie in $Z^+$. Thus they must both correspond to ends of
$Z^+$, and there are inequivalent proper rays representing these ends that both land at $p$.
\end{proof}

Although the circles $S^1_-,S^1,S^1_+$ are all canonically isomorphic,
the {\em orientations} on $S^1$ and $S^1_+$ agree, whereas they both {\em disagree} with
the orientation on $S^1_-$. This is true by convention; i.e.\/ this choice of orientation 
is part of the definition of $Z^\pm$.

\begin{lemma}[Circular order]\label{lemma:circular_order}
Let $f:S^1 \to S^2$ be a CaTherine wheel with zipper $Z^\pm$.
For $i=0,1,2$ let $p_i \in S^2$ be three distinct points with unique preimages 
$e_i \in S^1$ under $f$. Let $r_i$ be equivalence classes of
proper rays in $Z^+$ that land at $p_i$. Then the circular order on the $e_i$ in
$S^1$ agrees with the circular order on the ends of the $r_i$ in $Z^+$.
\end{lemma}
\begin{proof}
This is true by the definition of the circular order on the ends of $Z^+$, and the
natural (orientation-preserving) identification of $S^1_+$ with $S^1$ (the domain of $f$).
\end{proof}

\section{Checking the axioms for the LQG geodesic tree}
\label{sec-lqg-proof}

In this section we will deduce Theorem~\ref{thm-lqg-wheel} from Theorem~\ref{theorem:half_zipper}. 
Throughout, we fix $\gamma \in (0,2)$ and write $D$ for the $\gamma$-LQG metric, as introduced just after~\eqref{eqn-lqg-tensor}. For $t  > 0$ and $z\in\BB C$, we define the open LQG metric ball
\eqb  \label{eqn-lqg-ball}
\mcl B_t(z) := \left\{w\in\BB C : D(z,w) < t\right\}  .
\eqe 
We note that $\ol{\mcl B_t(z)}$ is compact~\cite[Proposition 3.8]{lqg-metric-estimates}. Moreover, the complement $\BB C\settminus \ol{\mcl B_t(z)}$ a.s.\ has infinitely many connected components~\cite[Theorem 1.14]{lqg-zero-one}.
For $r > 0$, we also define the Euclidean ball
\eqb \label{eqn-eucl-ball}
B_r(z) := \left\{w\in\BB C : |z-w| < r \right\} .
\eqe 

\subsection{Preliminary results on confluence of geodesics}
\label{sec-lqg-prelim} 

Our proof that $Z_\infty$ of Theorem~\ref{thm-lqg-wheel} satisfies the axioms in Theorem~\ref{theorem:half_zipper} is primarily based on existing results for LQG geodesics coming from~\cite{lqg-zero-one} (which in turn builds on~\cite{gm-confluence}). The most important input is the following proposition, which is an easy consequence of results in~\cite{lqg-zero-one}. See Figure~\ref{fig-finite-cross}, left for an illustration of the proposition statement.

\begin{proposition}[Confluence across a $D$-annulus] \label{prop-finite-cross}
Let $z\in\BB C$. Almost surely, for every choice of radii $s > t > 0$, there is a finite set of points $\mcl X = \mcl X_{t,s} \subset \bdy \mcl B_t(z)  $ such that the following is true. Every $D$-geodesic from $z$ to a point of $\BB C\settminus \ol{\mcl B_s(z)}$ passes through a point in $\mcl X$; and there is a unique $D$-geodesic from $z$ to $x$ for each $x\in\mcl X$.
\end{proposition}

\begin{figure}[t]
\begin{center}
\includegraphics[width=0.45\textwidth]{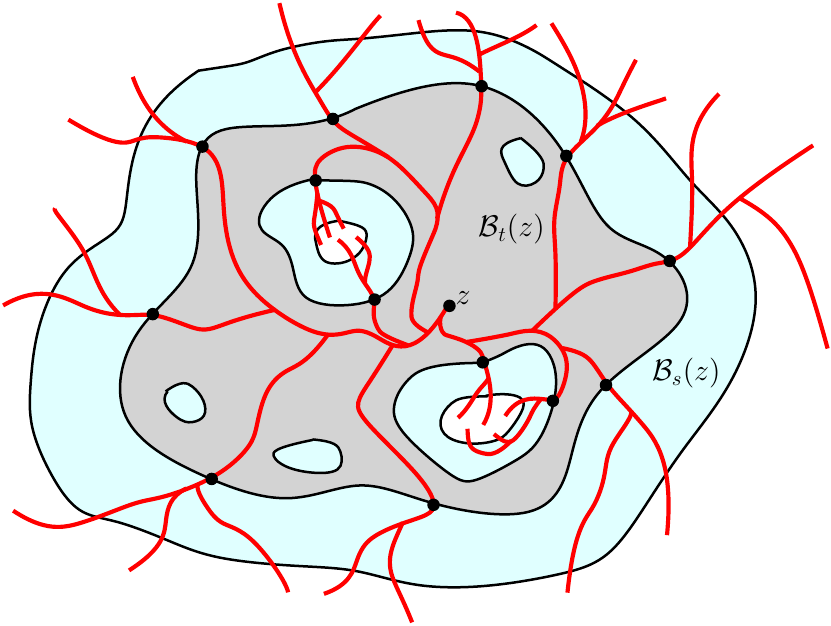}  
\includegraphics[width=0.35\textwidth]{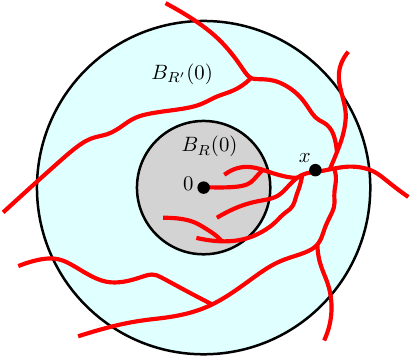}  
\caption{\label{fig-finite-cross} 
\textbf{Left:} Illustration of the statement of Proposition~\ref{prop-finite-cross}. Points of $\mcl X$ are shown in black and various representative $D$-geodesics from $z$ to points of $\BB C\settminus \ol{\mcl B_s(z)}$ are shown in red.
\textbf{Right:} Illustration of the statement of Lemma~\ref{lem-confluence-pt}. 
}
\end{center}
\end{figure}

\begin{proof} 
We first prove the proposition statement for a fixed deterministic pair of times $s > t > 0$. 
We claim that a.s.\ there are only finitely many connected components of $\BB C\settminus \ol{\mcl B_t(z)}$ which intersect $\BB C\settminus \ol{\mcl B_s(z)}$. 
Let us first note that each such connected component must also intersect $\bdy\mcl B_s(z)$, so it suffices to show that a.s.\ there are only finitely many connected components of $\BB C\settminus \ol{\mcl B_t(z)}$ which intersect $\bdy\mcl B_s(z)$. 
By~\cite[Lemma 3.8]{lqg-metric-estimates}, the closed $D$-ball $\ol{\mcl B_s(z)}$ is $D$-compact. If $U$ and $V$ are connected components of $\BB C\settminus \ol{\mcl B_t(z)}$ which intersect $\bdy \mcl B_s(z)$, then any path from $\bdy \mcl B_s(z) \cap U$ to $\bdy \mcl B_s(z) \cap V$ must pass from $\bdy \mcl B_s(z)$ to $\bdy\mcl B_t(z)$ and then from $\bdy\mcl B_t(z)$ to $\bdy\mcl B_s(z)$, so must have $D$-length at least $2(s-t)$. Hence $D(U,V)\geq 2(s-t)$. Since $\bdy\mcl B_s(z)$ can be covered by finitely many $D$-balls of radius $s-t$, there can be at most finitely many connected components of $\BB C\settminus \ol{\mcl B_t(z)}$ which intersect $\bdy\mcl B_s(z)$. 
 
For $w\in  \BB C $ with $D(z,w) > t$, define the filled metric ball $\mcl B_t^{w,\bullet}(z) := \BB C\settminus \mcl U$, where $\mcl U = \mcl U(t,z,w)$ is the connected component of $\BB C\settminus \ol{\mcl B_t(z)}$ which contains $w$. By the translation invariance of the law of $\Phi$, viewed modulo additive constant, we may apply~\cite[Propositions 3.6 and 3.7]{lqg-zero-one} to filled metric balls centered at $z$ (instead of at 0). These propositions imply that a.s.\ for each $w\in \BB Q^2 \settminus \{z\}$ with $D(z,w) > s$, there exists a finite set of points $\mcl X_{t,s}^w\subset \bdy\mcl B_t(z)$ such that the following is true (we restrict to $w\in\BB Q^2$ since~\cite[Proposition 3.7]{lqg-zero-one} holds a.s.\ for a fixed choice of $w$). Each $D$-geodesic from $z$ to a point of $\BB C\settminus \mcl B_s^{w,\bullet}(z)$ passes through a point of $\mcl X_{t,s}^w$, and there is a unique $D$-geodesic from $z$ to $x$ for each $x\in\mcl X_{t,s}^w$. 

By the first paragraph, we can find a finite set $W \subset\BB C$ which includes one point of $(\BB C\settminus \ol{\mcl B_s(z)})\cap U$ for each connected component $U$ of $\BB C\settminus \ol{\mcl B_t(z)}$ for which this intersection is non-empty. Then the proposition statement holds with $\mcl X =\bigcup_{w\in W} \mcl X_{t,s}^w \subset\bdy\mcl B_t(z)$. 

It remains to extend to a result which holds for all times $s > t > 0$ simultaneously. The proposition statement for deterministic $s > t > 0$ implies that the proposition statement a.s.\ holds simultaneously for every rational $s > t > 0$. Given an arbitrary $s > t > 0$, choose $s', t' \in \BB Q$ such that $s > s' > t' > t$. Let $\mcl X' = \mcl X_{t',s'} \subset \bdy\mcl B_{t'}(z)$ be as in the proposition statement. For $x\in \mcl X'$, let $\sigma_x : [0,t']\to\BB C$ be the (unique) $D$-geodesic from $z$ to $x$. By the defining properties of $\mcl X'$, each $D$-geodesic $\sigma$ from $z$ to a point of $\BB C\settminus \mcl B_t(z)$ (which is a subset of $\BB C\settminus \mcl B_{t'}(z)$) passes through some $x\in \mcl X'$. Hence, it must coincide with $\sigma_x$ on $[0,t']$. In particular, $\sigma(t) = \sigma_x(t)$. Since the $D$-geodesic from $z$ to $x$ is unique, so is the $D$-geodesic from $z$ to $\sigma_x(t)$. Hence, the proposition statement holds with $\mcl X =\mcl X_{s,t} = \{\sigma_x(t) : x\in\mcl X'\}$. 
\end{proof}
%By definition, every point which is hit by a geodesic started from $z$ at a time other than its terminal time belongs to $Z_z$. Hence the same property continues to hold if we remove from $\mcl X^0$ all of the points which are not in $Z_z$. That is, the proposition statement holds with $\mcl X = \mcl X^0 \cap Z_z$.  

We are primarily interested in the $D$-geodesic tree rooted at $\infty$, as defined in~\eqref{eqn-geo-tree-infty}.
But, for some lemmas it is more convenient to work with the {\em $D$-geodesic tree rooted at $z$} for $z\in\BB C$, which is the set
\eqb   \label{eqn-geo-tree-z} 
Z_z := \left\{ \sigma(t) \,:\, \text{$\sigma$ is a $D$-geodesic from $z$ to some $w\in\BB C$ and $0\leq t < D(z,w)$} \right\} .
\eqe 
We will now state some lemmas which allow us to transfer results between $Z_z$ and $Z_\infty$. We first need the following, which is essentially proven in~\cite{lqg-zero-one}. See Figure~\ref{fig-finite-cross}, right for an illustration. We recall from~\eqref{eqn-eucl-ball} that $B_r(\cdot)$ denotes a Euclidean ball.

\begin{lemma}[Confluence across a Euclidean annulus] \label{lem-confluence-pt}
Almost surely, for each $R > 0$, there exists $R'  >R$ (random) and $x \in B_{R'}(0) \settminus B_{ R}(0)$ such that every $D$-geodesic from a point of $B_{  R}(0)$ to a point of $\BB C\settminus B_{R'}(0)$ passes through $x$. Moreover, if $w\in\BB C$ and there is a $D$-geodesic $\sigma$ from $w$ to a point of $\BB C\settminus B_{R'}(0)$ which enters $B_R(0)$, then necessarily $w\in B_{R'}(0)$. 
\end{lemma}
\begin{proof}
The first assertion is shown in the proof of~\cite[Proposition 4.4]{lqg-zero-one}.
Now let $w\in\BB C$ and $\sigma$ be as in the second assertion and suppose for contradiction that $w\notin B_{R'}(0)$. By the defining property of $x $, the segment of $\sigma$ between $w$ and the first time it enters $B_R(0)$ must visit $x $. But, again by the defining property of $x $, the segment of $\sigma$ after the first time it enters $B_R(0)$ must also visit $x $. This is impossible since $\sigma$ is a simple path. Hence $w$ must be in $B_{R'}(0)$.  
\end{proof}

\begin{lemma}[Comparing $Z_\infty$ and $Z_z$] \label{lem-infty-to-z}
Almost surely, for each $R > 0$, there exists $z \in\BB Q^2$ such that $Z_\infty\cap B_R(0) = Z_z\cap B_R(0)$. In fact, we can arrange so that for every $D$-geodesic $\sigma$ from a point $w\in \BB C$ to $\infty$ which intersects $B_R(0)$, there exists a $D$-geodesic $\wt \sigma$ from $w$ to $z$ such that $\sigma(t) = \wt \sigma(t)$ for every $t \geq 0$ such that $\sigma(t) \in B_R(0)$. 
\end{lemma}
\begin{proof}
By the second assertion of Lemma~\ref{lem-confluence-pt}, a.s.\ for each $R > 0$ there exists $R_1 > R$ such that if $w\in\BB C$ and there is a $D$-geodesic from $w$ to a point of $\BB C\settminus B_{R_1}(0)$ which enters $B_R(0)$, then necessarily $w\in B_{R_1}(0)$. 
By the first assertion of Lemma~\ref{lem-confluence-pt} (with $R_1$ in place of $R$), there exists $R_2 > R_1$ and $x\in B_{R_2}(0)\settminus B_{R_1}(0)$ such that every $D$-geodesic from a point of $B_{R_1}(0)$ to a point of $\BB C\settminus B_{R_2}(0)$ passes through $x$. 

Let $z \in\BB Q^2\settminus B_{R_2}(0)$. We claim that the lemma statement holds for this choice of $z$. Indeed, let $w \in \BB C$ and let $\sigma $ be a $D$-geodesic from $w$ to $\infty$ which intersects $B_R(0)$. By the previous paragraph, $w\in B_{R_1}(0)$. By the defining property of $x $, $\sigma$ visits $x $. Furthermore, since $z\notin B_{R_2}(0)$, every $D$-geodesic from $w$ to $z$ also visits $x $. From this, we get that the segment of $\sigma$ between $w$ and $x $ is also a segment of a $D$-geodesic from $w$ to $z$. Let $\wt \sigma$ be this geodesic from $w$ to $z$. Then $\sigma(t) = \wt \sigma(t)$ for each $t \in [0, D(w, x )]$. The defining property of $x$ implies that $\sigma$ cannot re-visit $B_R(0)$ after hitting $x$, since otherwise $\sigma$ would have to visit $x$ a second time (which is impossible since geodesics are simple paths). Hence $\sigma(t) = \wt \sigma(t)$ for each $t\geq0 $ such that $\sigma(t) \in B_R(0)$. This implies the second assertion in the lemma statement and also that $Z_\infty\cap B_R(0) \subset Z_z \cap B_R(0)$. A similar argument gives $Z_z\cap B_R(0) \subset Z_\infty \cap B_R(0)$.
\end{proof}

\subsection{The LQG geodesic tree is a half-zipper}
\label{sec-tree-disconnect}

In this subsection, we will check that $Z_\infty$ is a half-zipper in the sense of Definition~\ref{definition:half_zipper}.  
We start with the following lemma for the LQG geodesic tree rooted at $z\in\BB C$, as defined in~\eqref{eqn-geo-tree-z}.

\begin{lemma}[Removing a point disconnects $Z_z$] \label{lem-tree-disconnect}
Let $z\in\BB C$. The following is true a.s. Let $\sigma : [0,T] \to \BB C$ be a $D$-geodesic started from $z$ and let $t \in (0,T)$. Then $Z_z \settminus \{\sigma(t)\}$ is not connected and $\sigma[0,t)$ and $\sigma(t,T)$ lie in different connected components of $Z_z\settminus \{\sigma(t)\}$. 
\end{lemma}

For the proof of Lemma~\ref{lem-tree-disconnect}, we need the following fact about $Z_z$. 

\begin{lemma}[Unique geodesics in $Z_z$ and $Z_\infty$]  \label{lem-tree-geo-unique}
Let $z\in\BB C$. Almost surely, for each $p \in Z_z$, there is a unique $D$-geodesic from $z$ to $p$. 
Furthermore, for each $p\in Z_\infty$, there is a unique $D$-geodesic from $p$ to $\infty$. 
\end{lemma}
\begin{proof}
By the definition~\eqref{eqn-geo-tree-z}, each $p\in Z_z$ is hit by a $D$-geodesic started from $z$ at a time other than its terminal time. The statement for $Z_z$ is therefore immediate from~\cite[Lemma 3.9]{lqg-zero-one}. 

To get the statement for $Z_\infty$, let $p\in Z_\infty$ and let $\sigma$ and $ \sigma'$ be $D$-geodesics from $p$ to $\infty$. By Proposition~\ref{prop-geo-infty}, there exists $R >0$ such that $\sigma\settminus B_R(0) = \sigma'\settminus B_R(0)$. Let $z \in\BB Q^2$ be as in Lemma~\ref{lem-infty-to-z} for this choice of $R$. Then $p \in Z_z$, so the first part of the present lemma statement tells us that there is a unique $D$-geodesic from $p$ to $z$. The second assertion of Lemma~\ref{lem-infty-to-z} (applied once to each of $\sigma$ and $\sigma'$) then tells us that $\sigma\cap B_R(0) = \sigma'\cap B_R(0)$, so $\sigma=\sigma'$. 
\end{proof}

\begin{figure}[t]
\begin{center}
\includegraphics[width=0.45\textwidth]{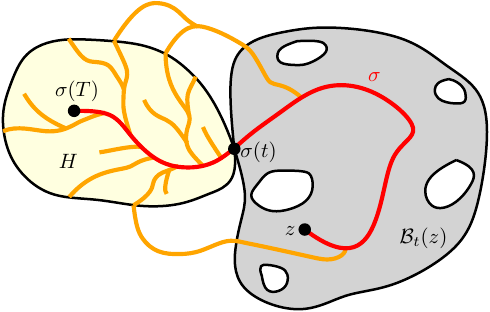}   
\caption{\label{fig-tree-disconnect} 
Illustration of the proof of Lemma~\ref{lem-tree-disconnect}. The geodesic $\sigma$ is shown in red, and other $D$-geodesics are shown in orange. The set $H$ of points which have a $D$-geodesic to $z$ which passes through $\sigma(t)$ is shown in light yellow. We show that each point of $\bdy H\settminus \{\sigma(t)\}$ has at least two distinct $D$-geodesics to $z$, one which passes through $\sigma(t)$ and one which passes through a different point of $\bdy\mcl B_t(z)$. By Lemma~\ref{lem-tree-geo-unique}, this implies that $\bdy H\settminus \{\sigma(t)\}$ is disjoint from $Z_z$. 
}
\end{center}
\end{figure}

\begin{proof}[Proof of Lemma~\ref{lem-tree-disconnect}]
See Figure~\ref{fig-tree-disconnect} for an illustration. 
Let $H = H_{\sigma(t)}$ be the set of points $w \in \BB C$ such that there exists a $D$-geodesic from $z$ to $w$ which passes through $\sigma(t)$ (necessarily at time $t = D(z,\sigma(t))$). Equivalently, 
\eqb \label{eqn-confluence-choice}
H = \left\{ w \in \BB C \,:\, D(w, \sigma(t)) = D(w,z) -t \right\} .
\eqe 
By~\eqref{eqn-confluence-choice} and since $D$ is continuous, the set $H$ is closed. 

We claim that $\bdy H \settminus \{\sigma(t)\}$ is disjoint from $Z_z$. Assuming the claim, we have that $(Z_z \settminus \{\sigma(t)\}) \cap H$ and $(Z_z \settminus \{\sigma(t)\} ) \settminus H$ are disjoint open subsets of $Z_z\settminus \{\sigma(t)\}$, the first set contains $\sigma(t,T)$, and the second set contains $\sigma[0,t)$, so the lemma statement follows. 

It remains to prove the above claim. By Lemma~\ref{lem-tree-geo-unique}, it suffices to show that for each $w \in \bdy H \settminus \{\sigma(t)\}$, there are at least two distinct $D$-geodesics from $z$ to $w$. For this purpose, let $w \in \bdy H \settminus \{\sigma(t)\}$. Since $H$ is closed, $w\in H$. There is a sequence of points $w_n \in \BB C\settminus H$ such that $w_n \to w$. By~\eqref{eqn-confluence-choice}, $D(w,z) - t = D(w,\sigma(t)) > 0$. Since $w_n\to w$, we can find $s \in (0,T)$ such that $D(w_n , z) > s$ for each sufficiently large $n\in\BB N$. 

Let $\mcl X = \mcl X_{t,s} \subset \bdy\mcl B_t(z)$ be the finite set as in Proposition~\ref{prop-finite-cross}, so every $D$-geodesic from $z$ to a point of $\BB C\settminus \mcl B_s(z)$ passes through a point of $\mcl X$. Since $\sigma$ passes through $\sigma(t)$, we have $\sigma(t) \in \mcl X$. For $n\in\BB N$, let $\sigma_{w_n}$ be a $D$-geodesic from $z$ to $w_n$. Since $w_n\notin H$, for each sufficiently large $n$, the point $\sigma_{w_n}(t)$ is a point of $\mcl X$ other than $\sigma(t)$. Since $\mcl X$ is finite, by possibly passing to a subsequence we can arrange so that there exists $x\in\mcl X \settminus \{\sigma(t)\}$ such that $\sigma_{w_n}(t) = x$ for every $n\in\BB N$. 

By taking $\sigma_{w_n}$ to be identically equal to $w_n$ for times larger than $D(z,w_n)$, we can arrange so that the paths $\sigma_{w_n}$ are defined on the same time interval. Since each $\sigma_{w_n}$ is a $D_{\Phi}$-geodesic and $D$ induces the Euclidean topology, the Arz\'ela-Ascoli theorem allows us to extract a subsequence of the paths $\sigma_{w_n}$ which converges uniformly to a $D$-geodesic $\sigma_w$ from $z$ to $w$. We have $\sigma_w(t) = x$. But, since $w\in H$, there is also a $D$-geodesic from $z$ to $w$ which passes through $\sigma(t) \not=x$ at time $t$. Hence, there are two distinct $D$-geodesics from $z$ to $w$, so by Lemma~\ref{lem-tree-geo-unique} we have $w\notin \bdy H$.  
\end{proof}

\begin{lemma}[Unique paths in $Z_\infty$] \label{lem-disconnect-infty}
The following is true a.s. Let $\sigma : [0,\infty)\to\BB C$ be a $D$-geodesic from some point to $\infty$ and let $t > 0$. Then there is no path from any point of $\sigma(0,t)$ to any point of $\sigma(t,\infty)$ in $Z_\infty\settminus \{\sigma(t)\}$. In particular, every point of $Z_\infty$ is a cut point and there is a unique path in $Z_\infty$ between any two points of $Z_\infty$, modulo time change. 
\end{lemma}
\begin{proof}
We work on the probability-one event that the conclusion of Lemma~\ref{lem-tree-disconnect} holds simultaneously for every $z\in\BB Q^2$, and the conclusion of Lemma~\ref{lem-infty-to-z} holds.  
Let $\eta$ be a path in $Z_\infty$ from a point of $\sigma(0,t)$ to a point of $\sigma(t,\infty)$. Then the image of $\eta$ (still denoted by $\eta$) is compact, so there exists $R > 0$ such that $\eta \subset B_R(0)$. 
By Lemma~\ref{lem-infty-to-z}, there exists $z \in\BB Q^2$ such that $Z_\infty\cap B_R(0) = Z_z\cap B_R(0)$. Moreover, there is a $D$-geodesic $\wt \sigma : [0,T] \to\BB C$ from $\sigma(0)$ to $z$ such that $\wt \sigma(t) = \sigma(t)$ for each $t > 0$ such that $\sigma(t) \in B_R(0)$. Then $\eta$ is a path in $Z_\infty\cap B_R(0) = Z_z\cap B_R(0)$ from $\wt \sigma(0,t)$ to $\wt \sigma(t,T)$. 
By Lemma~\ref{lem-tree-disconnect} for $z$, the path $\eta$ must pass through $\wt \sigma(t) = \sigma(t)$. Hence, there is no path from $\sigma(0,t)$ to $\sigma(t,\infty)$ in $Z_\infty\settminus \{\sigma(t)\}$, as required.

By definition, every point of $Z_\infty$ is of the form $\sigma(t)$ for some $\sigma$ and $t$ as in the lemma statement. Hence the first assertion of the lemma statement implies that every point of $Z_\infty$ is a cut point.  

Finally, we prove the uniqueness of paths. Let $p,q \in Z_\infty$ be distinct and let $\sigma_p$ and $\sigma_q$ be $D$-geodesics from $p$ and $q$ to $\infty$ (which are a.s.\ unique by Lemma~\ref{lem-disconnect-infty}). By Proposition~\ref{prop-geo-infty}, the paths $\sigma_p$ and $\sigma_q$ eventually meet and subsequently merge into each other, i.e., there exists $\tau_p , \tau_q \geq 0$ such that $\sigma_p(\tau_p+t) = \sigma_q(\tau_q+t)$ for every $t\geq0 $. Let $\tau_p$ and $\tau_q$ be the smallest times for which this is the case. We get a simple path $\sigma_{p,q}$ in $Z_\infty$ from $p$ to $q$ by following $\sigma_p|_{[0,\tau_p]}$, then the reverse of $\sigma_q|_{[0,\tau_q]}$. 

We will now argue that $\sigma_{p,q}$ is the unique (modulo time change) simple path in $Z_\infty$ from $p$ to $q$. Let $t\in (0,\tau_p)$. The first assertion of the lemma statement implies that $p$ and $\sigma_p(\tau_p) = \sigma_q(\tau_q)$ are in different path-connected components of $Z_\infty\settminus \{\sigma_p(t)\}$. Since $\sigma_q|_{[0,\tau_q]}$ is a path from $q$ to $\sigma_p(\tau_p)$ which does not visit $\sigma_p(t)$, also $p$ and $q$ are in different path-connected components of $Z_\infty\settminus \{\sigma(t)\}$. Hence every path in $Z_\infty$ from $p$ to $q$ must visit $\sigma(t)$. This holds for every $t\in (0,\tau_p)$, and the image of a path from $p$ to $q$ is closed, so any such path in $Z_\infty$ must visit each point of $\sigma_p[0,\tau_p]$. The same proof also applies with $q$ in place of $p$. Hence any path from $p$ to $q$ in $Z_\infty$ must contain the image of $\sigma_{p,q}$. If this path is simple, this implies that it is in fact a time change of $\sigma_{p,q}$.  
\end{proof}

\begin{lemma}[$Z_\infty$ is a half-zipper] \label{lem-half-zipper}
$Z_\infty$ is a half-zipper in the sense of Definition~\ref{definition:half_zipper}.
\end{lemma}
\begin{proof}
It is immediate from the definition~\eqref{eqn-geo-tree-infty} that $Z_\infty$ is dense in $\BB C$. By Lemma~\ref{lem-disconnect-infty}, there is a unique simple path in $Z_\infty$ between any two points of $Z_\infty$ (modulo time change) and every point of $Z_\infty$ is a cut point.
\end{proof}

\subsection{The LQG geodesic tree has short hair} 
\label{sec-lqg-short-hair}

To prove Theorem~\ref{thm-lqg-wheel}, it remains to prove the ``short hair'' property for $Z_\infty$. As in Section~\ref{sec-tree-disconnect}, we start by proving a closely related result for the tree $Z_z$ of~\eqref{eqn-geo-tree-z}.

\begin{lemma}[Confluence in a $D$-ball] \label{lem-finite-merge}
Let $z \in\BB C$. Almost surely, for each $T>0$ and each $\delta>  0$, there exists a finite set $\mcl Y\subset \mcl B_T(z) \cap Z_z$ such that the following is true. Each $D$-geodesic from a point $w\in \mcl B_T(z)$ to $z$ visits some $y\in \mcl Y$ at or before time $\delta$. 
\end{lemma}
\begin{proof}
We work on the probability-one event of Proposition~\ref{prop-finite-cross}. 
Let $N\in\BB N$ and let $0 = t_0 < t_1 < \dots < t_{n-1} < t_n = T$ be a partition of $[0,T]$ such that 
\eqbn
\max_{n\in [1,N]\cap\BB Z} (t_n - t_{n-1}) \leq \delta/2 . 
\eqen
For each $n = 2,\dots,N$, let $\mcl X_n = \mcl X_{t_{n-1} , t_n}$ be as in Proposition~\ref{prop-finite-cross} with $(t,s) = (t_{n-1} , t_n)$. 
By the definition~\eqref{eqn-geo-tree-z}, every point which is hit by a geodesic started from $z$ at a time other than its terminal time belongs to $Z_z$. Hence the defining property of $\mcl X_n$ from Proposition~\ref{prop-finite-cross} continues to hold if we remove from $\mcl X_n$ all of the points which are not in $Z_z$. We may therefore assume without loss of generality that $\mcl X_n\subset Z_z$. 

Let
\eqbn
\mcl Y := \{z\} \cup \bigcup_{n=2}^N \mcl X_n \subset \mcl B_T(z) \cap Z_z .
\eqen
We check the condition in the lemma statement for this choice of $\mcl Y$. 
Let $w \in \mcl B_T(z)$ and choose $n\in [1,N]\cap\BB Z$ such that $D(z,w) \in (t_{n-1}, t_n]$. If $n \leq 2$, then $D(w,z)  \leq t_2 \leq \delta$ so the desired property holds with $y = z$. Now suppose that $n\geq 3$. Since $w\in \BB C\settminus \ol{\mcl B_{t_{n-1}}(z)}$, the definition of $\mcl X_{n-2}$ implies that each $D$-geodesic $\sigma$ from $w$ to $z$ visits some $x\in \mcl X_{n-2}$. Since $D(x,z) = t_{n-2}$, the time at which $\sigma$ visits $x$ is $D(z,w) - t_{n-2} \leq t_n -t_{n-2} \leq \delta$. 
\end{proof}

We next transfer from a statement for $Z_z$ to a statement for $Z_\infty$ using Lemma~\ref{lem-confluence-pt}. 

\begin{lemma}[Confluence for $D$-geodesics to $\infty$] \label{lem-finite-infty}
Almost surely, for every $R > 0$ and $\delta > 0$, there is a finite set of points $\mcl Y \subset Z_\infty$ such that the following is true. Each $D$-geodesic from a point $w \in\BB C$ to $\infty$ which enters $B_R(0)$ visits some $y\in \mcl Y$ at or before time $\delta$.
\end{lemma}
\begin{proof} 
We start by choosing radii $R_2 > R_1 > R$, a point $z\in\BB Q^2$, and a radius $T>0$. 
By the second assertion of Lemma~\ref{lem-confluence-pt}, there exists $R_1  > R$ such that every $D$-geodesic from a point of $\BB C$ to $\infty$ which enters $B_R(0)$ must start at a point of $B_{R_1}(0)$. 
By~\cite[Lemma 3.8]{lqg-metric-estimates}, there exists $R_2 > R_1$ such that $D(w,z) > \delta$ for every $w \in B_{R_2}(0)$ and $z \in \BB C\settminus B_{R_1}(0)$. 
Let $z\in\BB Q^2$ be chosen so that the conclusion of Lemma~\ref{lem-infty-to-z} holds with $R_2$ in place of $R$, so that in particular $Z_\infty \cap B_{R_2}(0) = Z_z\cap B_{R_2}(0)$. 
Let $T  > 0$ be chosen so that $B_{R_2}(0) \subset \mcl B_T(z)$ (such a $T$ exists since $D$ induces the Euclidean topology on $\BB C$). 
 
Almost surely, the conclusion of Lemma~\ref{lem-finite-merge} holds simultaneously for every $z\in\BB Q^2$. Let $\wt{\mcl Y} \subset\mcl B_T(z) \cap Z_z$ be the finite set of points as in Lemma~\ref{lem-finite-merge} for $z$ and $T$ as in the first paragraph. Let $\mcl Y := \wt{\mcl Y} \cap B_{R_2}(0)$. 
Then $\mcl Y\subset Z_z\cap B_{R_2}(0) \subset Z_\infty$, where the inclusion is by our choice of $z$. 
We will check the condition in the lemma statement for this choice of $\mcl Y$. 
 
Let $w \in \BB C$ and let $\sigma $ be a $D$-geodesic from $w$ to $\infty$. 
Assume that $\sigma $ enters $B_R(0)$. 
Then our choice of $R_1$ gives $w \in B_{R_1}(0)$. 
By the second assertion of Lemma~\ref{lem-infty-to-z} and our choice of $z$, there is a $D$-geodesic $\wt\sigma$ from $w$ to $z$ such that $\sigma(t) = \wt\sigma(t)$ for each $t\geq 0$ such that $\sigma(t) \in B_{R_2}(0)$. 
By the defining property of $\wt{\mcl Y}$ from Lemma~\ref{lem-finite-merge}, the $D$-geodesic $\wt\sigma$ visits some $y\in \wt{\mcl Y}$ before time $\delta$. Since $w \in B_{R_1}(0)$ and by our choice of $R_2$, $\wt\sigma$ cannot exit $B_{R_2}(0)$ before time $\delta$. Therefore, $y\in B_{R_2}(0)$ so $y\in \mcl Y$. Furthermore, $\sigma$ and $\wt\sigma$ agree up until $\wt\sigma$ hits $y$, so $\sigma$ also hits $y$ before time $\delta$.  
\end{proof}

\begin{figure}[t]
\begin{center}
\includegraphics[width=0.45\textwidth]{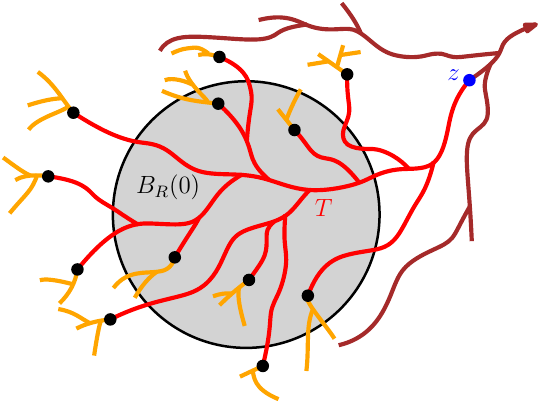}   
\caption{\label{fig-short-hair} 
Illustration of the proof of Lemma~\ref{lem-short-hair}. The finite set $\mcl Y \subset\BB C$ of Lemma~\ref{lem-finite-infty} is shown in black. The finite subtree $T$ is shown in red. Some representative points of $H_y$ for $y\in\mcl Y$ are shown in orange. Some representative points of $H_\infty$ (which admit $D$-geodesics to $\infty$ which do not visit any $y\in\mcl Y$) are shown in brown. 
}
\end{center}
\end{figure}

\begin{lemma}[$Z_\infty$ has short hair] \label{lem-short-hair}
Almost surely, $Z_\infty$ has short hair with respect to the spherical metric on $\BB C\cup\{\infty\}$ in the sense of Definition~\ref{definition:half_zipper}.
\end{lemma}
\begin{proof}
See Figure~\ref{fig-short-hair} for an illustration. 
Fix $\ep > 0$. We need to construct a compact set $T\subset Z_\infty$ with the topology of a finite tree such that every path-connected component of $Z_\infty\setminus T$ has spherical diameter at most $\ep$. 

Let $R > 0$ be large enough so that the spherical diameter of $\BB C\settminus B_R(0)$ is at most $\ep/2$.
Since $D$ induces the Euclidean topology, there exists $\delta >0$ such that for each $z,w\in B_R(0)$ with $D(z,w) \leq \delta$, the spherical distance from $z$ to $w$ is at most $\ep/2$. 

Let $\mcl Y\subset Z_\infty$ be the finite set of points as in Lemma~\ref{lem-finite-infty} with the above choice of $R$ and $\delta$. For $y\in\mcl Y$, let $\sigma_y$ be the $D$-geodesic from $y$ to $\infty$ (which is unique by Lemma~\ref{lem-tree-geo-unique}). Let $z \in \BB C \settminus B_R(0)$ be a point which lies on $\sigma_y$ for every $y\in \mcl Y$ (such a point exists by Proposition~\ref{prop-geo-infty}). Let $T$ be the union of the segment of $\sigma_y$ from $y$ to $z$, over all $y\in\mcl Y$. Then $T$ is a compact subset $Z_\infty$ homeomorphic to a finite tree. 

It remains to check the diameter condition for the path-connected components of $Z_\infty\settminus T$. For $y \in \mcl Y$, let $H_y$ be the set of $p\in Z_\infty \settminus T$ such that the $D$-geodesic from $p$ to $\infty$ passes through $p$ (this $D$-geodesic is unique by Lemma~\ref{lem-tree-geo-unique}). Also let $H_\infty$ be the set of $p\in Z_\infty\settminus T$ such that the $D$-geodesic from $p$ to $\infty$ does not pass through any $y\in\mcl Y$. 

If $p \in H_y$ and $q\in H_{y'}$ for some distinct $y,y' \in \mcl Y $, then there is a simple path in $Z_\infty$ from $p$ to $q$ which intersects $T$: this path follows a geodesic segment from $p$ to $y$, then the unique simple path in $T$ from $y$ to $y'$, then a geodesic segment from $y'$ to $q$. A similar argument shows that the same is true if one of $y$ or $y'$ is equal to $\infty$.
%replace geodesic segment with segments of geodesics from $p$ to $\infty$ and from $z$ to $\infty$. 
By Lemma~\ref{lem-disconnect-infty}, this path is in fact the unique simple path $Z_\infty$ between $p$ and $q$, so there is no path in $Z_\infty$ between $p$ and $q$ which does not intersect $T$. 
Hence every path-connected component of $Z_\infty\settminus T$ is contained in $H_y$ for some $y\in \mcl Y \cup \{\infty\}$. 

By the defining property of $\mcl Y$ from Lemma~\ref{lem-finite-infty}, for $y\in\mcl Y$ each point of $H_y \cap B_R(0)$ lies at $D$-distance at most $\delta$ from $y$. Hence the $D$-diameter of $H_y\cap B_R(0)$ is at most $\delta$, so (by our choice of $\delta$) the spherical diameter of $H_y\cap B_R(0)$ is at most $\ep/2$. By our choice of $R$, the spherical diameter of $\BB C\settminus B_R(0)$ is at most $\ep/2$, so by the triangle inequality the spherical diameter of $H_y$ is at most $\ep$ (note that $H_y\cup \{y\}$ is connected by definition). Again using the property of $\mcl Y$ from Lemma~\ref{lem-finite-infty}, we have $H_\infty\subset \BB C\settminus B_R(0)$. Hence our choice of $R$ implies that the spherical diameter of $H_\infty$ is at most $\ep$. Therefore every path-connected component of $Z_\infty\settminus T$ has spherical diameter at most $\ep$, as required.
\end{proof}

\begin{proof}[Proof of Theorem~\ref{thm-lqg-wheel}]
By Lemmas~\ref{lem-half-zipper} and~\ref{lem-short-hair}, $Z_\infty$ is a half-zipper with short hair (Definition~\ref{definition:half_zipper}). Hence we can apply Theorem~\ref{theorem:half_zipper} (with $Z_\infty$ viewed as a subset of the Riemann sphere $\BB C\cup\{\infty\}$) to get the desired CaTherine wheel $f : S^1\to \BB C\cup\{\infty\}$. 
\end{proof}

\subsection{Properties of the LQG CaTherine wheel}
\label{sec-lqg-order}

In this subsection we will prove Theorem~\ref{thm-order}. 
We first need the following lemma, which is a straightforward consequence of Lemma~\ref{lemma:unique_preimage}. 

\begin{lemma}[Unique preimages for LQG] \label{lem-lqg-preimage}
Let $f : S^1\to\BB C\cup\{\infty\}$ be as in Theorem~\ref{thm-lqg-wheel}. 
For each $w\in\BB C \setminus Z_\infty$ for which there is a unique $D$-geodesic from $w$ to $\infty$, the set $f^{-1}(w)$ consists of a single point. 
Moreover, $f^{-1}(\infty)$ is a single point. 
\end{lemma} 
\begin{proof}
Recall Definition~\ref{definition:landing} of a ray in $Z_\infty$ and its landing point. 
By Lemma \ref{lemma:unique_preimage}, a point in $(\BB C\cup \{\infty\}) \setminus Z_\infty$ has a unique preimage under $f$ if and only if it is the landing point of a unique equivalence class of rays in $Z_\infty$ (where rays are declared to be equivalent if they agree outside of a compact subset of $Z_\infty$). We will check that this holds for $\infty$ and for each $w\in \BB C\setminus Z_\infty$ for which there is a unique $D$-geodesic from $w$ to $\infty$. 

By iteratively applying Lemma~\ref{lem-confluence-pt}, we obtain a sequence of points $Z_\infty \ni x_n \to \infty$ such that each $D$-geodesic from a point of $\BB C$ to $\infty$ passes through infinitely many of the $x_n$s. By the definition~\eqref{eqn-geo-tree-infty} of $Z_\infty$, this implies that every unbounded path in $Z_\infty$ passes through infinitely many of the $x_n$s. 

If $\eta ,\wt\eta : [0,\infty) \to Z_\infty$ are simple paths in $Z_\infty$ going to $\infty$, then each of $\eta$ and $\wt\eta$ visit infinitely many of the $x_n$s. By Lemma~\ref{lem-disconnect-infty}, for each $n,m\in\BB N$ with $n <m$, there is a unique simple path in $Z_\infty$ from $x_n$ to $x_m$, modulo time change. Hence the images of $\eta$ and $\wt\eta$ coincide outside of a compact set. Therefore $\eta$ and $\wt\eta$ belong to the same equivalence class of rays, so $ \infty$ is the landing point of a unique equivalence class of rays. 

Now let $w\in \BB C \setminus Z_\infty$ and assume that there is a unique $D$-geodesic $\sigma_w$ from $w$ to $\infty$. Then $\sigma_w|_{(0,1]}$ is a ray in $Z_\infty$ which lands at $w$ (at time 0). We claim that any other ray in $Z_\infty$ is equivalent to this one. Indeed, consider an arbitrary ray whose landing point is $w$, represented by a simple path $\eta : (0,1] \to Z_\infty$ which extends continuously to $\eta : [0,1]\to \BB C$ with $\eta(0) = z$. Since $w \notin Z_\infty$ the $D$-geodesic from $\eta(1)$ to $\infty$ does not visit $w$, so since the range of this $D$-geodesic is closed it is disjoint from $\eta[t,\infty)$ for some $t > 0$.  
By concatenating $\eta$ with the $D$-geodesic from $\eta(1)$ to $\infty$, then erasing any loops which the concatenated path makes, we therefore obtain a function $\wt\eta: [0,\infty) \to \BB C$ with $\wt\eta(0,\infty)\subset Z_\infty$ and $\wt\eta(t) \to \infty$ as $t\to\infty$ which coincides with $\eta$ for small enough times. Hence the rays $\wt\eta|_{(0,T]}$ and $\eta$ are equivalent for each $T>0$. 

The path $\wt\eta$ visits $x_n$ for each sufficiently large $n$. By Lemma~\ref{lem-disconnect-infty}, for each $t > 0$ and each sufficiently large $n$, the segment of $\wt\eta$ from $\wt\eta(t)$ until the first time it visits $x_n$ coincides (modulo time change) with the unique simple path in $Z_\infty$ from $\wt\eta(t)$ to $x_n$. Since this holds for every $n$, we get that $\wt\eta|_{[t,\infty)}$ agrees (modulo time change) with the $D$-geodesic from $\wt\eta(t)$ to $\infty$. Sending $t\to 0$ shows that $\wt\eta$ coincides (modulo time change) with a $D$-geodesic from $w$ to $\infty$. Hence $\wt\eta$ and $\sigma_w$ agree modulo time change. In particular, for any $T>0$, $\wt\eta|_{(0,T]}$ and $\sigma_w|_{(0,1]}$ are equivalent rays. 
\end{proof}

\begin{proof}[Proof of Theorem~\ref{thm-order}]
By Lemma~\ref{lem-lqg-preimage}, $\infty$ has a unique preimage under $f$. 
The LQG area measure $\mu$ assigns positive mass to every open set and zero mass to every point (see, e.g.,~\cite[Theorem 2.1]{bp-lqg-notes}). Identify $S^1\settminus f^{-1}(\infty)$ with the open interval $(-1,1)$, so that $f : (-1,1) \to\BB C$. By re-parametrizing, we can assume that $f(0) = 0$. Let $\phi : (-1,1) \to \BB R$ be defined so that $\phi(0) = 0$, $\phi(t) = \mu(f[0,t])$ for $t \in (0,1)$, and $\phi(t) = -\mu(f[ t,0])$ for $t\in (-1,0)$. By the definition of a CaTherine wheel, the sets $f[a,b]$ for $-1<a<b<1$ have non-empty interior, hence positive $\mu$-mass. Hence $\phi$ is strictly increasing. Since $f$ is continuous, for any $t \in (-1,1)$, the intersection $\bigcap_{\ep > 0} f[t-\ep,t+\ep]$ is a single point. Since $\mu$ assigns zero mass to each point, we have 
\eqbn
\lim_{\ep\to 0} ( \phi(t+\ep) - \phi(t-\ep)) = \lim_{\ep\to 0}  \mu(f[t-\ep,t+\ep]) =0. 
\eqen
Hence $\phi$ is continuous, so $\phi$ is a homeomorphism. Thus, $g := f\circ\phi$ is continuous. 

Lemma~\ref{lem-lqg-preimage} implies that each point of $\BB C$ which has a unique $D$-geodesic to $\infty$ has a unique preimage under $f$, and hence also under $g$. The statement about the order in which $z$ and $w$ are hit by $g$ follows directly from Lemma~\ref{lemma:circular_order} applied with $(p_0,p_1,p_2) = (\infty, z, w)$. 
%Each $p\in Z_\infty$ corresponds to at least two ideal gaps in the setting of Section~\ref{sec-universal-circle}, one to the left side of a geodesic which visits $p$ and one to the right side of this same geodesic. Hence the definition of $f$ implies that each such point has at least two preimages under $f$ (equivalently, under $g$). 
\end{proof}

\bibliography{cibib,danny_refs}
\bibliographystyle{hmralphaabbrv}

\end{document}